\documentclass[aop,preprint]{imsart}

\RequirePackage[OT1]{fontenc}
\RequirePackage{amsthm,amsmath,amssymb,enumerate}
\RequirePackage[numbers]{natbib}
\RequirePackage[colorlinks,citecolor=blue,urlcolor=blue]{hyperref}
\usepackage{algorithm}
\usepackage{algcompatible}
\usepackage{caption}
\usepackage{subcaption}
\usepackage{graphicx}
\usepackage{geometry}
\usepackage{array}
\startlocaldefs
\numberwithin{equation}{section}
\theoremstyle{plain}
\newtheorem{thm}{Theorem}%[section]
\newtheorem{lemma}{Lemma}
\newtheorem{corollary}{Corollary}				
\newtheorem{proposition}{Proposition}
\theoremstyle{definition}
\newtheorem{definition}{Definition}
\newtheorem{assumption}{Assumption}
\newtheorem{remark}{Remark}

\def\one{{\bf 1}\hskip-.5mm}

\def\R{{\mathbb R}}

\def\N{{\mathbb N}}

\def\cF {\mathcal{F}}
\def\diam{\operatorname{Diam}}
\def\supp{\operatorname{Supp}}

\endlocaldefs

\begin{document}

\begin{frontmatter}
\title{Mean Field Limits for Nonlinear Spatially Extended Hawkes Processes with Exponential Memory Kernels}
\runtitle{}

\begin{aug}
\author{\fnms{J.} \snm{Chevallier}\thanksref{m1}\ead[label=e1]{julien.chevallier@u-cergy.fr}},
\author{\fnms{A.} \snm{Duarte}\thanksref{m2}\ead[label=e2]{aline.duart@gmail.com}},
\author{\fnms{E.} \snm{L\"ocherbach}\thanksref{m1}\ead[label=e3]{eva.loecherbach@u-cergy.fr}}
\and
\author{\fnms{G.} \snm{Ost}\thanksref{m2}
\ead[label=e4]{guilhermeost@gmail.com}
\ead[label=u1,url]{http://www.foo.com}}

\thankstext{t1}{\today}
%\thankstext{t2}{}
%\thankstext{t3}{}
%\thankstext{t4}{}
\runauthor{J. Chevallier et al.}

\affiliation{Universit\'e de Cergy-Pontoise\thanksmark{m1} and  Universidade de S\~ao Paulo\thanksmark{m2}}

%\address{A. Duarte\\
%E. L\"ocherbach\\
%G. Ost\\
%Universit\'e de Cergy-Pontoise\\ 
%AGM, CNRS-UMR 8088,\\
%95000 Cergy-Pontoise,\\
%France\\
%\printead{e1}\\
%\printead{e2}\\
%\printead{e3}}

\end{aug}

\begin{abstract}
We consider spatially extended systems of interacting nonlinear Hawkes
processes modeling large systems of neurons placed in $\R^d$ and study the associated mean field limits. As the total number of neurons tends to infinity, we prove that the evolution of a typical neuron, attached to a given spatial position, can be described by a nonlinear limit differential equation driven by a Poisson random measure. The limit process is described by a neural field equation. As a consequence, we provide a rigorous derivation of the neural field equation based on a thorough mean field analysis. 
\end{abstract}

\begin{keyword}[class=MSC]
\kwd[]{60G55}
\kwd{60G57}
\kwd{60K35}
\end{keyword}

\begin{keyword}
\kwd{Hawkes processes}
\kwd{Spatial mean field}
\kwd{Propagation of Chaos}
\kwd{Neural field equation}
%\kwd{Longtime Behavior}
\kwd{Coupling}
\end{keyword}

\end{frontmatter}

\section{Introduction}
The aim of this paper is to present a microscopic model describing a large network of spatially structured interacting neurons, and to study its large population limits. Each neuron is  placed in $\R^d$. Its activity is represented by a point process accounting for the successive times at which the neuron emits an action potential, commonly  referred to as a spike. The firing intensity of a neuron depends on the past history of the neuron. Moreover, this intensity is affected by the activity of other neurons in the network. Neurons interact mostly through chemical synapses. A spike of a pre-synaptic neuron leads to a change in the membrane potential of the post-synaptic neuron (namely an increase if the synapse is excitatory or a decrease if the synapse is inhibitory), possibly after some delay. In neurophysiological terms this is called synaptic integration. Thus, excitatory inputs from the neurons in the network increase the firing intensity, and inhibitory inputs decrease it. Hawkes processes provide good models for this synaptic integration mechanism by the structure of their intensity processes, see \eqref{eq:intensity0} below. We refer to \cite{ccdr}, \cite{chorno},  \cite{hrbr} and to \cite{pat} for the use of Hawkes processes in neuronal modeling. For an overview of point processes used as stochastic models for interacting neurons both in discrete and in continuous time and related issues, see also \cite{gl}.

In this paper, we study spatially structured systems of interacting Hawkes processes representing the time occurrences of action potentials of neurons. Each neuron is characterized by its spike train, and the whole system is described by  the multivariate counting process $(Z^{(N)}_{1}(t), \ldots, Z^{(N)}_{N}(t))_{t\geq 0}$.
%$Z^{(N)}_{i} ( t), t \geq 0, 1 \le i \le N.$ 
Here, the integer $N\geq 1$ stands for the size of the neuronal network and $ Z^{(N)}_{ i} ( t)$ represents the number of spikes of the $i$th neuron in the network during the time interval $ [0, t ].$ This neuron is placed in a position $x_i \in \R^d, $ and we assume that the empirical distribution of the positions $\mu^{(N)}(dx):=N^{-1}\sum_{i=1}^{N}\delta_{x_{i}}(dx)$ converges\footnote{Convergence in the sense of the Wasserstein $W_{2}$-distance is considered in Scenario (\hyperref[ass:deterministic:position]{$S_{2}$}) below.} to some probability measure $ \rho(dx)$ on $\R^d$ as $N\to \infty$.

The multivariate counting process $(Z^{(N)}_{1}( t), \ldots, Z^{(N)}_{N}( t))_{t\geq 0}$ is characterized by its intensity process $ (\lambda^{(N)}_{ 1} ( t) ,\ldots, \lambda^{(N)}_{ N} ( t)  )_{t\geq 0} $ (informally) defined through the relation 
$$ P\left ( Z^{(N)}_{ i } \mbox{ has a jump in } ]t , t + dt ] | {\cal F}_t \right) = \lambda^{(N)}_{i } (t)  dt , $$
where $ {\cal F}_t = \sigma (  Z^{(N)}_{ i} (s)  , \, s \le t ,{1 \le i \le N} ) .$ We work with a spatially structured network of neurons in which $ \lambda^{(N)}_{ i } (t) $ is given by $\lambda^{(N)}_i(t)=f( U^{(N)}_{i}(t- ) )$ with
\begin{equation}\label{eq:intensity0}
U^{(N)}_{i}(t) := e^{-\alpha t}u_{0}(x_{i})+ \frac{1}{N} \sum_{j=1}^{N} w\left(x_{j},x_{i}\right)\int_{]0, t ] }e^{-\alpha(t-s)}dZ^{(N)}_j(s).
\end{equation}
Here, $f : \R \to \R_+$ is the {\it firing rate function} of each neuron and $w:\R^d\times \R^d\to\R$  is the matrix of {\it synaptic strengths };  for each $i, j \in\{1,\ldots, N\} $, the value $ w(x_j, x_i )$ models the influence of neuron $j$ (located in position $x_j$) on neuron $i$ (in position $x_i$). 
The parameter $\alpha\geq 0$ is the {\it leakage rate.} Moreover, $u_0 ( x_i) $ is the initial input to the membrane potential of neuron $i.$ 

Equation (\ref{eq:intensity0}) has the typical form of the intensity of a multivariate nonlinear Hawkes process, going back to \cite{Hawkes1971} and \cite{ho} who introduced Hawkes processes in a univariate and linear framework. We refer to \cite{bm} for the stability properties of multivariate nonlinear Hawkes processes, and to \cite{dfh} and \cite{chevallier} for the study of Hawkes processes in high dimensions.  

In this paper, we study the limit behavior of the system $(Z^{(N)}_{1} ( t),\ldots, Z^{(N)}_{1} ( t) )_{ t\geq 0} $ as $N\to\infty$. Our main result states that -- under suitable regularity assumptions on the parameters $u_0, w $ and $f$ -- the system can be approximated by a system of inhomogeneous independent Poisson processes $(\bar Z_x (t ))_{t\geq 0} $ associated with positions $ x \in \R^d $ which can informally be described as follows. In the limit system, the spatial positions of the neurons are distributed according to the probability measure $ \rho ( dx)$. 
%(since $\mu^{(N)}(dx)$ converges to $\rho(dx)$).
Given a position $x\in\R^d$, the law of the attached process $(\bar Z_x (t ))_{t\geq 0} $ is the law of an inhomogeneous Poisson process having intensity given by $(\lambda ( t, x ))_{t\geq 0}$. Here $ \lambda ( t, x ) = f (u(t,x) ) $ and $u(t, x) $ solves the {\it scalar neural field equation}
\begin{equation}\label{eq:neural:field}
\begin{cases}
\displaystyle \dfrac{\partial u(t,x)}{\partial t}=-\alpha u(t,x)+ \int_{\R^d}w(y,x)f(u(t,y))\rho(dy),\\
u(0,x)= u_{0}(x).
\end{cases}
\end{equation}

Such scalar neural field equations (or neural field models) have been studied extensively in the literature, see e.g.\ \cite{Wilson_Cowan:72}, \cite{Wilson:73} and \cite{Amari:77}.
They constitute an important example of spatially structured neuronal networks with nonlocal interactions, see \cite{Bressloff:12} for a recent and comprehensive review. Let us cite a remark made by Paul Bressloff, on page 15 of \cite{Bressloff:12} : 

{\it There does not currently exist a multi-scale analysis of conductance-based neural networks that provides a rigorous derivation of neural field equations, although some progress has been made in this direction [...]. One crucial step in the derivation of neural field equations presented here was the assumption of slowly varying synaptic currents, which is related to the assumption that there is not significant coherent activity at the level of individual spikes. This allowed us to treat the output of a neuron (or population of neurons) as an instantaneous firing rate. A more rigorous derivation would need to incorporate the mean field analysis of local populations of stochastic spiking neurons into a larger scale cortical model, and to carry out a systematic form of coarse graining or homogenization in order to generate a continuum neural field model.} 

Our model is \emph{not} a conductance-based neural network. Nevertheless, with the present work, to the best of our knowledge, we present a first rigorous derivation of the well-known neural field equation as mean field limit of spatially structured Hawkes processes. 

The paper is organized as follows. In Section \ref{sec:def} we introduce the model and provide an important a priori result, stated in Proposition \ref{Prop:2}, on the expected number of spikes of a typical neuron in the finite size system. In Section \ref{sec:3} we present our main results, Theorems \ref{thm:1} and \ref{thm:2}, on the convergence of spatially extended nonlinear Hawkes processes towards the neural field equation \eqref{eq:neural:field}.
This convergence is expressed in terms of the empirical measure of the spike counting processes associated with the neurons as well as the empirical measure correspondent to their position. Therefore, we work with probability measures on the space $D([0,T],\N)\times \R^d$ and a convenient distance defined on this space which is introduced in  \eqref{def:distance}. The main ingredient of our approach is the unique existence of the limit process (or equivalently, of its intensity process), which is stated in Proposition \ref{prop:uniqueness}. Once the unique existence of the limit process is granted, our proof makes use of a suitable coupling technique for jump processes which has already been applied in \cite{dfh}, \cite{chevallier} and \cite{SusanneEva}, together with a control of the Wasserstein distance of order $2$ between the empirical distribution of the positions $\mu^{(N)}(dx)$ and the limit distribution $\rho(dx).$  

As a by product of these convergence results, we obtain Corollaries \ref{cor:propagation:of:chaos} and \ref{cor:convergencetoneuralfield}, the former being closely connected to the classical propagation of chaos property and the latter stating the convergence of the process $(U^{(N)}_1(t),\ldots, U^{(N)}_N(t))_{t\geq 0}$ towards to the solution of the neural field equation.  

In Section \ref{sec:4} are given the main technical estimates we use. Sections \ref{sec:5} and \ref{sec:6} are devoted to the proofs of our two main results, Theorem \ref{thm:1} and \ref{thm:2}, together with Corollary \ref{cor:convergencetoneuralfield}. In Section \ref{Sec:Discussion_on_the_parameters} we discuss briefly the assumptions that are imposed on the parameters of the model by our approach. Finally, some auxiliary results are postponed to 
%Appendix 
\ref{sec:App}.

\section{General Notation, Model Definition and First Results}\label{sec:def}
\subsection{General notation}
Let $(S,d)$ be a Polish metric space and $\mathcal{S}$ be the Borel $\sigma$-algebra of $S$. The supremum norm of any real-valued $\mathcal{S}$-measurable function $h$ defined on $S$ will be denoted by $\|h\|_{S,\infty}:=\sup_{x\in S}|h(x)|$. We will often write $\|h\|_{\infty}$ instead of $\|h\|_{S,\infty}$ when there is no risk of ambiguity. 
For any real-valued function $h(y,x) $ defined on $S\times S$ and $x\in S$, the $x$-section of $h$ is the function  $h^x$ defined on $S$ by $h^x(y)=h(y,x)$ for all $y\in S$. Similarly,  for any $y\in S$, the $y$-section of  $h$ is the function  $h_y$ defined on $S$ by $h_y(x)=h(y,x)$ for all $x\in S.$
The space of all continuous functions from $(S,d)$ to $(\R_+ ,|\cdot|)$ will be denoted by  $C(S,\R_+ ).$ 
For any measure $\nu$ on $(S,\mathcal{S})$ and $\mathcal{S}$-measurable function $h:S\to \R$, we shall write $\left\langle h,\nu \right\rangle=\int_S h(x)\nu(dx)$ when convenient.  
For any $p \geq 1, $ we shall write $L^p(S,\nu)$ to denote the space of $\mathcal{S}$-measurable functions $h: S\to \R$  such that $\| h \|_{L^p ( \nu ) } := (\int |h(x)|^p d \nu(x)  )^{1/p} <\infty$. 

For two probability measures $\nu_1$ and $\nu_2$ on $(S,\mathcal{S})$, the Wasserstein distance of order $p$ between $\nu_1$ and $\nu_2$ associated with the metric $d$ is defined as
$$
W_p(\nu_1,\nu_2)=\inf_{\pi\in\Pi(\nu_1,\nu_2)}\left( \int_S\int_Sd(x,y)^p \pi(dx,dy) \right)^{1/p} ,
$$
where $\pi$ varies over the set $\Pi(\nu_1,\nu_2)$ of all probability measures on the product space $S\times S$ with marginals $\nu_1$ and $\nu_2$. Notice that  the Wasserstein distance of order $p$ between $\nu_1$ and $\nu_2$ can be rewritten as  the infimum of $E[d(X,Y)^p]^{1/p}$ over all possible couplings $(X,Y)$ of the random elements $X$ and $Y$ distributed according to $\nu_1$ and $\nu_2$ respectively, i.e.
$$
W_p(\nu_1,\nu_2)=\inf\left\{E[d(X,Y)^p]^{1/p}: \mathcal{L}(X)=\nu_1\ \mbox{and} \ \mathcal{L}(Y)=\nu_2  \right\}.
$$

Let $Lip(S)$ denote the space of all real-valued Lipschitz functions on $S$ and
$$
\|h\|_{Lip}=\sup_{x\neq y} \frac{|h(x)-h(y)|}{d(x,y)}.
$$
We write $Lip_1(S)$ to denote the subset of $Lip(S)$ such that $\|h\|_{Lip}\leq 1.$
When $p=1,$ the Kantorovich-Rubinstein duality provides another useful representation for $W_1(\nu_1,\nu_2)$, namely
$$
W_1(\nu_1,\nu_2)=\sup_{h\in Lip_1(S)\cap L^1\left(S,d|\nu_1-\nu_2|\right) }\left\{\int_{S}h(x)(\nu_1-\nu_2)(dx)\right\}.
$$
Furthermore, the value of the supremum above is not changed if we impose the extra condition that $h$ is bounded, see e.g.\ Theorem 1.14 of \cite{Villanibook}.

\subsection{The model and preliminary remarks}
Throughout this paper we work on a filtered probability space $(\Omega ,\cF , (\cF_t)_{t\geq 0} , Q  )  $ which is rich enough such that all following processes may be defined on it. 
%In particular, we will suppose that on  $(\Omega ,\cF , \mathbb{F} , Q  )  $ may be defined a sequence  $(\Pi_i(dz,ds))_{i\geq 1}$ of i.i.d.\ Poisson random measures with intensity measure $dsdz$ on $\R_+\times \R_+$. 
We consider a system of interacting nonlinear Hawkes processes which is spatially structured. In the sequel, the integer $N\geq 1$ will denote the number of processes in the system. Each process models the behavior of a specific neuron. All neurons are associated with a given spatial position belonging to $ \R^{d}$. These spatial positions are denoted by $x_{1},\dots,x_{N}.$ In the following, $\R^d  $ will be equipped with a fixed norm $ \| \cdot \| .$  The positions $x_{1},\dots,x_{N}$ are assumed to be fixed in this section (unlike Section \ref{sec:5} where the positions are assumed to be random).

We now describe the dynamics of the $N$ Hawkes processes associated with these positions. 

\begin{definition}
\label{def:1}
Let $f:\R\to \R_+$, $w:\R^d\times \R^d \to \R$ and $u_{0}:\R^d \to \R$ be measurable functions and  $\alpha\geq 0$ be a fixed parameter. A $(\cF_t)_{t\geq 0}$-adapted multivariate counting process $(Z_1^{(N)} (t) ,\ldots, Z_1^{(N)} (t) )_{t\geq 0}$ defined on $(\Omega,\cF, (\cF_t)_{t\geq 0}, Q)$ is  said to be a multivariate Hawkes process with parameters $(N,f,w,u_{0},\alpha),$ associated with the positions $ (x_1, \ldots , x_N) , $ if
\begin{enumerate}
\item $Q-$almost surely, for all pairs $i,j\in \{1,\dots,N\}$ with  $i\neq j$, the counting processes $(Z^{(N)}_{i}(t))_{t\geq 0}$ and $(Z^{(N)}_{j}(t))_{t\geq 0}$ never jump simultaneously.
\item For each $i\in  \{1,\dots,N\}$ and $t\geq 0$, the compensator of $Z^{(N)}_{i}(t)$ is given by  $\int_{0}^t\lambda^{(N)}_i(s )ds$
where $(\lambda^{(N)}_i(t))_{t\geq 0 }$ is the non-negative $(\cF_t)_{t\geq 0}-$progressively measurable process defined, for all $t\geq 0$, by $\lambda^{(N)}_i(t)=f( U^{(N)}_{i}(t- ) )$ with
\begin{equation}\label{eq:intensity}
U^{(N)}_{i}(t)= e^{-\alpha t}u_{0}(x_{i})+ \frac{1}{N} \sum_{j=1}^{N} w\left(x_{j},x_{i}\right)\int_{]0, t ] }e^{-\alpha(t-s)}dZ^{(N)}_j(s).
\end{equation}
\end{enumerate}
\end{definition}

\begin{remark}
\label{rmk:1}
Notice that for each $i\in\{1,\ldots, N\}$, the process $(U^{(N)}_i(t))_{t\geq 0}$ satisfies the following stochastic differential equation
\begin{equation*}
d U^{(N)}_i(t) = -\alpha U^{(N)}_{i}(t) dt + \frac{1}{N} \sum_{j=1}^{N} w\left(x_{j},x_{i}\right) dZ^{(N)}_j(t).
\end{equation*}
\end{remark}

The functions $f : \R \to \R_+ $  and $w:\R^d \times \R^d \to \R$  are called {\it spike rate function} and {\it matrix of synaptic strengths} respectively.
The parameter $\alpha\geq 0$ is called the {\it leakage rate.} For each neuron $i\in \{1,\dots,N\}$, $u_{0}(x_{i})$ is interpreted as an initial input to the spike rate of neuron $i$ (its value depends on the position $x_{i}$ of the neuron). \footnote{Without too much effort, the initial input $u_{0}(x_{i})$ could be replaced by a random input of the form
$ \frac{1}{N} \sum_{j=1}^{N} U_{i,j} $, where the random variables $U_{i,1}, \ldots , U_{i,N}$ are i.i.d. distributed according to some probability measure $\nu (x_{i}, du )$ defined on $\R^d$.
In the limit $N\to +\infty$, we have the correspondence 
$\lim_{N\to +\infty} N^{-1} \sum_{j=1}^{N} U_{i,j}= \int u \nu (x_{i}, du) = u_{0}(x_{i})$.
}

An alternative definition of multivariate Hawkes processes which will be used later on is the following.

\begin{definition}
\label{def:2}
Let $(\Pi_i(dz,ds))_{1\leq i\leq N}$ be a sequence of i.i.d.\ Poisson random measures with intensity measure $dsdz$ on $\R_+\times \R_+$. A $(\cF_t)_{t\geq 0}$-adapted multivariate counting process $(Z_1^{(N)}(t), \ldots, Z_N^{(N)}(t))_{t\geq 0}$ defined on $(\Omega,\cF, (\cF_t)_{t\geq 0}, Q)$ is said to be a multivariate Hawkes process with parameters $(N,f,w,u_{0},\alpha),$ associated with the positions $ (x_1, \ldots , x_N) , $ if  $Q-$almost surely, for all $t\geq 0$ and $i\in\{1,\ldots, N\}$,
\begin{equation}
Z^{(N)}_i(t)=\int_{0}^{t}\int_{0}^{\infty}\one_{\left\{z\leq f\left( U^{(N)}_i(s- ) \right)\right\}}
\Pi_i(dz,ds),
\end{equation}
where $U^{N}_i(s)$ is given in \eqref{eq:intensity}.
\end{definition}

\begin{proposition}
\label{prop:equal_definition}
Definitions \ref{def:1} and \ref{def:2} are equivalent. 
\end{proposition}
We refer the reader to Proposition 3 of \cite{dfh} for a proof of Proposition \ref{prop:equal_definition}.
In what follows we will work under 

\begin{assumption}\label{ass:1}
The function $f$ is Lipschitz continuous with Lipschitz norm $L_f>0.$ 
\end{assumption}

\begin{assumption}
\label{Ass:2}
The initial condition $u_{0}$ is Lipschitz continuous with Lipschitz norm $ L_{u_0} > 0 $ and bounded, i.e., $ \| u_0 \|_{\R^d,\infty}<\infty.$
\end{assumption}

\begin{proposition}\label{prop:existence}
Under Assumption \ref{ass:1}, there exists a path-wise unique multivariate Hawkes process with parameters $(N,f,w,u_{0},\alpha)$ such that $t \mapsto \sup_{ 1 \le i \le N} E ( Z_i^{(N)} (t) ) $ is locally bounded.
\end{proposition}
The proof of Proposition \ref{prop:existence} relies on classical Picard iteration arguments and can be found (in a more general framework) in Theorem 6 of \cite{dfh}. 

The next result provides an upper bound for the expected number of jumps in a finite time interval $[0,T]$, for each  fixed $T>0$, which will be crucial later. 

\begin{proposition}
\label{Prop:2}
Under Assumptions \ref{ass:1} and \ref{Ass:2}, for each $N\geq 1$ and $T>0$, the following inequalities hold.    
\begin{equation}
\frac{1}{N}\sum_{i=1}^N E\left[(Z^{(N)}_i(T)) \right]\leq T\left(f(0)+L_f \|u_0\|_{\R^d,\infty} \right)\exp\left\{TL_f  \sup_{j}\|w_{x_j}\|_{L^1(\mu^{(N)})} \right\},
\end{equation}
and
\begin{multline}
\frac{1}{N}\sum_{i=1}^N E\left[(Z^{(N)}_i(T))^2 \right]\leq  \exp \left\{  T \left( 1 + 4 L_f^2 \left(\sup_{j}\|w_{x_j}\|_{L^2(\mu^{(N)})} \right)^2 \right) \right\} \times  \\
\times \left[ T\left(f(0)+L_f \| u_0\|_{\R^d,\infty} \right)\exp\left\{TL_f \sup_{j}\|w_{x_j}\|_{L^1(\mu^{(N)})}\right\} + 2 T f(0)^2 + 4 L_f^2 T\| u_0\|_{\R^d,\infty}^2 \right] 
,
\end{multline}
where $\mu^{(N)}(dx)=N^{-1}\sum_{i=1}^N\delta_{x_i}(dx)$ is the empirical distribution associated with the fixed spatial positions $x_1, \ldots, x_N\in \mathbb{R}^d.$
\end{proposition}
The proof of Proposition \ref{Prop:2} will be given in 
%Appendix 
\ref{proof of Prop.2}.

\section{Convergence of Spatially Extended Hawkes processes}\label{sec:3}
In this section, we present two convergence results for the empirical process of nonlinear spatially extended Hawkes processes, our main results. We fix a probability measure $ \rho(dx)$  on $( \R^d , {\cal B} ( \R^d ) ) $; it is the expected limit of the empirical distribution $\mu^{(N)}(dx)=N^{-1}\sum_{i=1}^{N}\delta_{x_{i}}(dx).$ The following additional set of assumptions will be required as well. 

\begin{assumption}\label{ass:exponential:moments}
\begin{sloppypar}
The measure $ \rho(dx)$ admits exponential moments, i.e., there exists a parameter $ \beta > 0$  such that $\mathcal{E}_{\beta}:=\int e^{\beta \|x\| }\rho(dx)<\infty$. 
\end{sloppypar}
\end{assumption}

\begin{assumption}
\label{ass:w:lipschitz}
The matrix of synaptic strengths $w$ satisfies the following Lipschitz and boundedness conditions. \\
1. There exists a constant $L_w>0$ such that for all  $x,x', y, y'\in \R^d,$ 
\begin{equation}\label{eq:lipschitz:assumption:w}
| w(y, x ) - w(y', x' ) | \le L_w ( \| x - x' \| +  \| y- y' \| )  .
\end{equation}
2. There exist $x_{0}$ and $y_{0}$ in $\mathbb{R}^{d}$ such that $||w(y_{0},\cdot)||_{\R^d,\infty}<+\infty$ and $||w(\cdot,x_{0})||_{\R^d,\infty}<+\infty$. 
\end{assumption}

\begin{remark}\label{rem:2} 
Notice that the vectors $x_{0}$ and $y_{0}$ can be taken as the origin $0^{d}$ in $\mathbb{R}^{d}$.  This follows directly from the Lipschitz continuity of $w$. Moreover, under Assumption \ref{ass:exponential:moments}, notice that Assumption \ref{ass:w:lipschitz}-2. is equivalent to the fact that the $y$-sections $w_y$ and the $x$-sections $w^x$ of $w$ are uniformly square integrable with respect to $\rho(dx)$, i.e.,
\begin{equation}\label{eq:assumption:w:2}
\sup_{ y \in \R^d } \int | w(y, x) |^2 \rho (dx) < \infty,
\end{equation} 
and
\begin{equation}\label{eq:assumption:w:2_x_fixed}
\sup_{ x \in \R^d } \int | w(y, x) |^2 \rho (dy) < \infty.
\end{equation}
In fact, \eqref{eq:assumption:w:2} and \eqref{eq:assumption:w:2_x_fixed} are the assumptions that naturally appear through the proofs (see e.g. \eqref{eq:uniform:control:discrete:L2} below where we need \eqref{eq:assumption:w:2}, and \eqref{eq:lal} where we need \eqref{eq:assumption:w:2_x_fixed}). 

\begin{proof}[Proof of the equivalence of Ass. \ref{ass:w:lipschitz}-2. with \eqref{eq:assumption:w:2} and \eqref{eq:assumption:w:2_x_fixed}, under Ass. \ref{ass:exponential:moments} and \eqref{eq:lipschitz:assumption:w}]\

We will only show that $\|w_{0^d}\|_{\R^d,\infty}$ is finite  if and only if \eqref{eq:assumption:w:2_x_fixed} holds, the other case is treated similarly. Observe that the Lipschitz continuity of $w$ implies that for each $x,y\in\R^d$,
$$
(\max\{0,|w_{0^d}(x)|-L_w\|y\|\})^2\leq |w^x(y)|^2\leq 2\left(\|w_{0^d}\|^2_{\R^d,\infty}+L^2_w\|y\|^2\right).
$$
As a consequence of the inequality above, it follows that
\begin{multline*}
\frac{|w_{0^d}(x)|^2}{4}\rho\left(B\Big( 0^d,\frac{|w_{0^d}(x)|}{2L_w} \Big) \right) \\
\leq \int |w^x(y)|^2\rho(dy)\leq 2\left(\|w_{0^d}\|^2_{\R^d,\infty}+L^2_w\int \|y\|^2\rho(dy)\right),
\end{multline*}
where for each $z\in\R^d$ and $r>0$, $B(z,r)$ is the ball of radius $r$ centered at $z$. Now, taking the supremum with respect to $x,$ we deduce that
\begin{equation}\label{eq:firsteq}
\sup_{x\in\R^d} \int |w^x(y)|^2\rho(dy)\leq 2\left(\|w_{0^d}\|^2_{\R^d,\infty}+L^2_w\int \|y\|^2\rho(dy)\right)
\end{equation}
and
\begin{equation}\label{eq:seceq}
\frac{\|w_{0^d}\|_{\R^d,\infty}^2}{4}\rho(B(0^d,\|w_{0^d}\|_{\R^d , \infty }/2L_w))\leq \sup_{x\in\R^d} \int |w^x(y)|^2\rho(dy).
\end{equation}
Under Assumption \ref{ass:exponential:moments}, we know that $\int \|y\|^2\rho(dy)$ is finite. Therefore, \eqref{eq:firsteq} shows that $\|w_{0^d}\|^2_{\R^d,\infty} < \infty $ implies \eqref{eq:assumption:w:2_x_fixed}. On the other hand,  \eqref{eq:seceq} shows that \eqref{eq:assumption:w:2_x_fixed} implies $\|w_{0^d}\|^2_{\R^d,\infty} < \infty $   if $ \rho(B(0^d,\|w_{0^d}\|_{\R^d , \infty }/2L_w)) > 0 .$ Finally, if $ \rho(B(0^d,\|w_{0^d}\|_{\R^d , \infty }/2L_w)) = 0 ,$ then it trivially holds that $\|w_{0^d}\|^2_{\R^d,\infty} < \infty $ since $L_w > 0 $ and since $ \rho(B(0^d, \infty)) = 1 .$ 
\end{proof}
  
\end{remark}
 
Let us now introduce, for each $T>0$ fixed,  the empirical measure 
\begin{equation}
\label{def:marginal_empirical_measure}
P^{(N,N)}_{[0,T]}( d\eta,dx)= \frac{1}{N}\sum_{i=1}^{N}\delta_{\left((Z^{(N)}_i(t))_{0\leq t\leq T}, x_i\right)}(d\eta,dx).
\end{equation}
This measure $P^{(N,N)}_{[0,T]}$ is a random probability measure  on the space $D([0,T],\N)\times \R^d ,$ where  $D([0,T],\N)$ is the space of c\`adl\`ag functions defined on the interval $[0,T]$ taking values in $\N$. The measure is random since it depends on the realization  of the $N$ counting processes $ Z_1^{(N)}, \ldots, Z_N^{(N)} , $ defined on $(\Omega ,\cF , (\cF_t)_{t\geq 0}, Q )$.

Under two frameworks described below, our two main results state that $P^{(N,N)}_{[0,T]}$ converges as $N\to\infty$ to a deterministic probability measure $P_{[0,T]}$ defined on the space $D([0,T],\N)\times \R^d$.  This convergence is stated with respect to a convenient distance between random probability measures $P$ and $\tilde P $ on the space $D([0,T],\N)\times \R^d$ which is introduced now.  

For c\`adl\`ad functions $\eta, \xi\in D([0,T],\N)$ we consider the distance $d_S(\eta,\xi)$ defined by
\begin{equation}
\label{def:skorohod_like_dist}
d_S(\eta,\xi)=\inf_{\phi\in I}\left\{\|\phi\|_{[0,T],*}\vee \|\eta-\xi(\phi)\|_{[0,T],\infty}\right\},
\end{equation}
where $I$ is the set of non-decreasing functions $\phi:[0,T]\to [0,T]$ satisfying $\phi(0)=0$ and $\phi(T)=T$ and where for any function $\phi\in I$ the norm $\|\phi\|_{[0,T],*}$ is defined as
$$
\|\phi\|_{[0,T],*}=\sup_{0\leq s<t\leq T}\log\left(\frac{\phi(t)-\phi(s)}{t-s}\right). 
$$
The metric $d_S(\cdot,\cdot)$ is equivalent to the classical Skorokhod distance. More importantly the metric space $(D([0,T],\N),d_S)$ is Polish, see for instance \cite{Billingsley:68}.

Finally, for any random probability measures $P$ and $ \tilde P$ on $D([0,T],\N)\times \R^d,$ we define the Kantorovich-Rubinstein like distance between $P$ and $\tilde P$  as
\begin{equation}
\label{def:distance}
d_{KR}(P, \tilde P )=\sup_{g\in Lip_1(D([0,T]\times \N)\times \R^d)}{ E }\left[|\left\langle g,P-\tilde P  \right\rangle |\right],
\end{equation}
where we recall that $\langle g,P-\tilde P \rangle =\int_{\R^d}\int_{D([0,T],\N)} g(\eta,x)(P-\tilde P )(d\eta,dx)$. Here the expectation is taken with respect to\ the probability measure $Q $ on $(\Omega ,\cF , (\cF_t)_{t\geq 0} )$, that is with respect to\ the randomness present in the jumps of the process.

Our convergence results are valid under two scenarios.
\begin{definition}\label{def:scenarios}
Consider the two following assumptions:

\smallskip
\newlength{\asslabel}
\setlength{\asslabel}{1.3cm}
\newlength{\assdef}
\setlength{\assdef}{13cm}
\newcommand{\spacenot}{\vspace{-0.2cm}}
\noindent
\begin{tabular}{>{\centering}p{\asslabel}|p{\assdef}}
\phantomsection
\label{ass:random:position} 
\spacenot $\left(S_{1}\right)$:   & \textbf{Random spatial distribution.}

The positions $x_{1},\dots,x_{N}$ are the realizations of an i.i.d.\ sequence of random variables $X_1, \ldots , X_N, \ldots , $ distributed according to $\rho(dx)$.
\end{tabular}

\medskip
\noindent
\begin{tabular}{>{\centering}p{\asslabel}|p{\assdef}}
\phantomsection
\label{ass:deterministic:position} 
\spacenot $\left(S_{2}\right)$:   & \textbf{Deterministic spatial distribution.}

The positions $x_{1},\dots,x_{N}$ are deterministic (depending on $\rho(dx)$) such that the sequence of associated empirical measures $\mu^{(N)}(dx)=N^{-1}\sum_{i=1}^{N}\delta_{x_{i}}(dx)$ satisfies $W_{2}(\mu^{(N)},\rho) \le K N^{-1/d'} $ for any $ d' > d $ and for all $N$ sufficiently large, for a fixed constant $ K > 0.$
\end{tabular}
\end{definition}

\begin{remark}
Under Scenario (\hyperref[ass:random:position]{$S_{1}$}), the random positions $x_{1},\dots,x_{N}$ are interpreted as a random environment for our dynamics. This random environment is supposed to be independent of the Poisson random measures $\Pi_{i}(dz,ds) , i \geq 1. $

In Scenario (\hyperref[ass:deterministic:position]{$S_{2}$}), the bound $N^{-1/d'}$ is reasonable compared to the generic optimal quantization rate $N^{-1/d}$ \cite[Theorem 6.2]{graf2007foundations}. This bound is used afterwards to control the contribution of the spatial approximation in our spatial mean field approximation context. The construction of a sequence satisfying the claimed bound is described in Section \ref{sec:6}.
\end{remark}

In Theorems \ref{thm:1} and \ref{thm:2} below we will prove that as $ N \to \infty$,
\begin{equation}\label{eq:28}
d_{KR} ( P^{(N,N)}_{[0,T]}, P_{[0, T ] } ) \to 0, 
\end{equation}
\emph{almost surely with respect to the random environment under Scenario (\hyperref[ass:random:position]{$S_{1}$})}.

Let us now discuss briefly the properties that this limit $P_{[0,T]}$ should {\it a priori} satisfy. First, we obtain the following integrability property.

\begin{proposition}\label{prop:apriori}
Under Assumptions \ref{ass:1}, \ref{Ass:2}, \ref{ass:w:lipschitz} and either Scenario (\hyperref[ass:random:position]{$S_{1}$}) or (\hyperref[ass:deterministic:position]{$S_{2}$}), the limit measure $P_{[0, T] } $ a priori satisfies that 
\begin{equation}\label{eq:apriori}
\int_{\R^d} \int_{D([0,T],\N)}  [ \eta_T^2 + \eta_T ] P_{[0, T ]} ( d \eta , d x ) < \infty . 
\end{equation}
%In particular, 
%$$ \int_{\R^d }  \left[ \left( \int_0^T \lambda (t, x) dt \right)^2 +  \int_0^T \lambda (t, x) dt  \right] \rho ( dx)< \infty $$
%for all $ T > 0 .$ 
\end{proposition}

Proposition \ref{prop:apriori} is the limit version of Proposition \ref{Prop:2}, and its proof is postponed to %Appendix 
\ref{proof of Propapriori}. 

Let us now precisely define the limit measure $ P_{[0, T ]}.$ Firstly, consider any real-valued smooth test function $(\eta,x)\mapsto g(\eta,x) \equiv g(x) $  defined on $D([0,T],\N)\times \R^d$   which does not depend on the variable $\eta$. Evaluating the integral of $g$ with respect to the probability measure $P^{(N,N)}_{[0,T]}$, and then letting $N\to \infty$, one deduces that the second marginal of $P_{[0,T]}$ on $\R^d$ must be equal to the probability measure $\rho(dx)$ (since $\mu^{(N)}(dx)$ converges to $\rho(dx)$).

Since $(D([0,T],\N),d_{S})$ is Polish, it follows from the Disintegration Theorem that $P_{[0,T]}$ can be rewritten  as 
\begin{equation}\label{eq:def:limit:measure}
P_{[0,T]}(d\eta,dx)=P_{[0,T]}(d\eta|x)\rho(dx),
\end{equation} 
where $P_{[0,T]}(d\eta|x)$ denotes the conditional distribution of $P_{[0,T]}$ given the position $x\in \R^d$. Here, $x\mapsto P_{[0,T]}(d\eta|x)$  is Borel-measurable in the sense that $ x \mapsto P_{[0, T]} ( A | x ) $ is Borel-measurable for any measurable subset $ A \subset D([0,T],\N).$
From a heuristic which is explained below, the conditional distribution $P_{[0,T]}(d\eta|x)$ (for each $x\in\supp(\rho)$, the support of $\rho$) turns out to be the law of a inhomogeneous Poisson point process with intensity process  $(\lambda(t,x))_{0\leq t\leq T}$  
where $\lambda=(\lambda(t,x) , {0\leq t\leq T,x\in \R^d}) $ is solution of the nonlinear equation 
\begin{equation}
\label{def:limit_intensity}
\lambda(t,x)=f\left( e^{-\alpha t}u_{0}(x) +\int_{\R^d}w(y,x)\int_{0}^te^{-\alpha(t-s)}\lambda(s,y)ds \rho(dy)\right).
\end{equation}
The heuristic relies on the following argument: at the limit, we expect that the firing rate at time $t$ of the neurons near location $y$ should be approximately equal to $\lambda(t,y)$. Taking this into account, Equation \eqref{def:limit_intensity} is the limit version of the interaction structure of our system described in Definition \ref{def:1}. In particular, the empirical mean with respect to the positions, i.e. the integral with respect to $\mu^{(N)}(dx)$, is replaced by an integral with respect to $\rho(dx)$.

Rewriting \eqref{eq:apriori} in terms of the intensities, we obtain moreover that {\it a priori}, 
$$ \int_{\R^d }  \left[ \left( \int_0^T \lambda (t, x) dt \right)^2 +  \int_0^T \lambda (t, x) dt  \right] \rho ( dx)< \infty , $$
for each fixed $ T > 0 .$ 
Existence and uniqueness of solutions for the nonlinear equation \eqref{def:limit_intensity} is now ensured by
\begin{proposition}
\label{prop:uniqueness}
Under Assumptions \ref{ass:1}, \ref{Ass:2} and \ref{ass:w:lipschitz},  for any solution $\lambda$ of the equation \eqref{def:limit_intensity} such that $t\mapsto \int_{\R^d}\int_{0}^t\lambda(s,y)ds\rho(dy)$ and $t\mapsto \int_{\R^d}\left( \int_{0}^t\lambda(s,y)ds \right)^2 \rho(dy)$ are locally bounded,  the following assertions hold. \\
1. For any $ T > 0 , $ $\lambda\in C([0,T]\times \R^d,\R_+)$ and $\|\lambda\|_{\R^d\times [0,T],\infty} < \infty .$ \\
2. 
$ \lambda $ is Lipschitz-continuous in the space variable, that is, there exists a positive constant $C=C(f,u_0,w,\alpha,\| \lambda \|_{ [0, T] \times \R^d, \infty }, T)$ such that for all $x,y\in \R^d$ and for all $ t \le T , $ 
\begin{equation}
\label{lambda_holder_L1}
|\lambda(t,x)-\lambda(t,y)| \leq C\|x-y\|. 
\end{equation}
Furthermore, if we consider the map $F$ from $C([0,T]\times \R^d,\R_+)$ to itself defined by 
$$
\lambda \mapsto F(\lambda)(t,x)=f\left( e^{-\alpha t}u_{0}(x) +\int_{\R^d}w(y,x)\int_{0}^te^{-\alpha(t-s)}\lambda(s,y)ds \rho(dy)\right) ,
$$
then for any $\lambda, \tilde{\lambda}\in C([0,T]\times \R^d,\R_+) $, the following inequality holds  
\begin{equation}
\label{ineq:uniqueness}
\|F(\lambda)-F(\tilde{\lambda})\|_{[0,T]\times \R^d,\infty}\leq \left(1-e^{-\alpha T}\right)\alpha^{-1}L_f \sup_{x\in \R^d} \|w^x \|_{L^1 (\rho)  } \|\lambda-\tilde{\lambda}\|_{[0,T]\times \R^d ,\infty}.
\end{equation}

\end{proposition} 
The proof of Proposition \ref{prop:uniqueness} is postponed to %Appendix 
\ref{proof of Prop.3}. The inequality \eqref{ineq:uniqueness} together with a classical fixed point argument imply not only the existence but also the uniqueness of a solution of the equation \eqref{def:limit_intensity}.

We are now in position to state the two main results of the present paper.

\begin{thm}
\label{thm:1}
Under Assumptions \ref{ass:1}--\ref{ass:w:lipschitz} and Scenario (\hyperref[ass:random:position]{$S_{1}$}),  there exists a positive constant $C = C(T,f,w,u_{0},\alpha,\beta )$ and a random variable $N_{0} $ depending only on the realization of the random positions $ X_1, \ldots, X_N, \ldots, $ such that for all $N\geq N_{0}$,
\begin{equation}
\label{bound_on_the_distance}
{d_{KR}\left(P^{(N,N)}_{[0,T]},P_{[0,T]}\right)} \leq C   \left( N^{- 1/2 } +   W_2 ( \mu^{(N)} , \rho ) \right).
\end{equation}
Moreover, if $\|w\|_{\R^d\times \R^d,\infty} < \infty, $ then $ N_{0}=1$ is valid.
Furthermore, for any fixed $ d' > d, $ it holds that  
\begin{equation}
\label{eq:secondbound} {d_{KR}\left(P^{(N,N)}_{[0,T]},P_{[0,T]}\right)}  \le C  N^{ - \frac{1}{ 4 + d'}} 
\end{equation}
eventually almost surely as $ N \to \infty. $ 
\end{thm}

\begin{thm}
\label{thm:2}
Under Assumptions \ref{ass:1}--\ref{ass:w:lipschitz} and Scenario (\hyperref[ass:deterministic:position]{$S_{2}$}), for each $T>0$, there exists a positive constant $C = C(T,f,w,u_{0},\alpha,\beta )$ such that for all $N\in \N$,
\begin{equation}
\label{eq:bound_on_the_distance:deterministic}
{d_{KR}\left(P^{(N,N)}_{[0,T]},P_{[0,T]}\right)} \leq C   \left( N^{- 1/2 } +   W_2 ( \mu^{(N)} , \rho ) \right).
\end{equation}
Furthermore, for any fixed $ d' >2\vee d, $ it holds that  
\begin{equation}
\label{eq:secondbound:deterministic} {d_{KR}\left(P^{(N,N)}_{[0,T]},P_{[0,T]}\right)}  \le C  N^{-\frac{1}{d'}}
\end{equation}
eventually as $N\to\infty$.
\end{thm}
The proofs of Theorems \ref{thm:1} and \ref{thm:2} are respectively postponed to Sections \ref{sec:5} and \ref{sec:6}. Here are given two corollaries that are valid under either Scenario (\hyperref[ass:random:position]{$S_{1}$}) or Scenario (\hyperref[ass:deterministic:position]{$S_{2}$}).

\begin{corollary}\label{cor:propagation:of:chaos}
1. For any bounded functions $g,\tilde{g}\in Lip_1$,
\begin{multline*}
\left|E\left[\left\langle g, P^{(N,N)}_{[0,T]}\right\rangle \left\langle \tilde{g}, P^{(N,N)}_{[0,T]}\right\rangle \right]-\left\langle g, P_{[0,T]}\right\rangle \left\langle \tilde{g}, P_{[0,T]}\right\rangle\right|\\
\leq \left( \|g\|_{\infty} + \|\tilde{g}\|_{\infty} \right) d_{KR}\left(P^{(N,N)}_{[0,T]},P_{[0,T]}\right).
\end{multline*}
2. Suppose moreover that $\rho$ admits a $\mathcal{C}^{1}$ density, denoted by $f_{\rho}$, with respect to the Lebesgue measure. Then for any $1$-Lipschitz bounded functions $\varphi,\tilde{\varphi}:D([0,T],\N) \to [-1,1]$ and any $x,\tilde{x}\in \mathbb{R}^{d}$ such that $f_{\rho}(x)\neq 0$ and $f_{\rho}(\tilde{x})\neq 0$, there exist kernel functions $ \Phi_N ( z) , \tilde \Phi_N ( z) $ with $ \Phi_N ( z)  f_\rho ( z) dz \stackrel{w}{\to } \delta_x (dz) $ and $ \tilde \Phi_N ( z)  f_\rho ( z) dz \stackrel{w}{\to } \delta_{\tilde{x}} (dz) ,$ as $N \to \infty , $ such that for 
%some sequences $(g_{N})_{N\geq 1}$ and $(\tilde{g}_{N})_{N\geq 1}$ of mappings from $D([0,T],\N)\times \mathbb{R}^{d}$ to $\mathbb{R}$ given by 
$$ g_N (\eta, z)  = \varphi (\eta) \Phi_N ( z)  \mbox{ and }  \tilde g_N (\eta, z) = \tilde \varphi (\eta ) \tilde \Phi_N ( z) , $$
we have
\begin{equation}
\label{Propagation_of_chaos}
E\left[\left\langle g_{N}, P^{(N,N)}_{[0,T]}\right\rangle \left\langle \tilde{g}_{N}, P^{(N,N)}_{[0,T]}\right\rangle \right] \to \left< \varphi, P_{[0,T]}(\cdot|x) \right> \left< \tilde{\varphi}, P_{[0,T]}(\cdot|\tilde{x}) \right>.
\end{equation}
\end{corollary}
\begin{remark}\label{rmk:propagation:of:chaos}
Equation \eqref{Propagation_of_chaos} has to be compared to the property of propagation of chaos for standard or multi-class mean field approximations. Thanks to a suitable spatial scaling, which is contained in the test functions $g_{N}$ and $\tilde{g}_{N}$ and is explicit in the proof, we find that the activity near position $x$ is asymptotically independent of the activity near position $\tilde{x}$. Relating this result with multi-class propagation of chaos as defined in \cite{graham2008chaoticity} for instance, let us mention that:
\begin{itemize}
\item if $x=\tilde{x}$, one recovers the chaoticity within a class,
\item if $x\neq \tilde{x}$, one recovers the chaoticity between two different classes.
\end{itemize}
The functions $\Phi_{N}$ and $\tilde{\Phi}_{N}$ are approximations to the identity with spatial scaling $N^{p(d)}$. The optimal scaling obviously depends on the scenario under study. Under Scenario (\hyperref[ass:random:position]{$S_{1}$}), we need $p(d)<[(4+d)(2d+1)]^{-1}$ whereas (\hyperref[ass:deterministic:position]{$S_{2}$}) the weaker condition $p(d)<[(2\vee d)(2d+1)]^{-1}$ is needed (the condition is weaker since the convergence is faster).
\end{remark} 
\begin{proof}[Proof of Corollary \ref{cor:propagation:of:chaos}]
For any bounded functions $g,\tilde{g}\in Lip_1$, we apply the triangle inequality to obtain the following bound
\begin{multline}
\left|E\left[\left\langle g, P^{(N,N)}_{[0,T]}\right\rangle \left\langle \tilde{g}, P^{(N,N)}_{[0,T]}\right\rangle \right]-\left\langle g, P_{[0,T]}\right\rangle \left\langle \tilde{g}, P_{[0,T]}\right\rangle\right| \\ 
\leq \|g\|_{\infty} \left|E\left[ \left\langle \tilde{g}, P^{(N,N)}_{[0,T]}-P_{[0,T]}\right\rangle \right] \right| +\|\tilde{g}\|_{\infty}\left| E\left[ \left\langle g, P^{(N,N)}_{[0,T]}-P_{[0,T]}\right\rangle \right] \right|
\end{multline}
from which we deduce the first statement.

Next, to construct suitable sequences $(g_{N})_{N\geq 1}$ and $(\tilde{g}_{N})_{N\geq 1}$, let us first give some control for the density $f_{\rho}$. Since $f_{\rho}$ is $\mathcal{C}^{1}$, let $r>0$, $\varepsilon>0$ and $M>0$ be such that for all $y$ in $B(x,r)\cup B(\tilde{x},r)$, $f_{\rho}(y) \geq \varepsilon$ (recall that $f_{\rho}(x)\neq 0,$ $f_{\rho}(\tilde{x})\neq 0$) and $\| \nabla f_{\rho}(y)\| \leq M$. 

Then, let $\Phi:\mathbb{R}^{d}\to \mathbb{R}$ be a mollifier, that is a compactly supported smooth function such that $\int \Phi(y) dy=1$. Let us define $\Phi_{N}$ and $\tilde{\Phi}_{N}$ by
\begin{equation*}
\Phi_{N}(y) = N^{dp(d)} \frac{\Phi(N^{p(d)}(y-x))}{f_{\rho}(y)} \ \text{ and } \ \tilde{\Phi}_{N}(y) = N^{dp(d)} \frac{\Phi(N^{p(d)}(y-\tilde{x}))}{f_{\rho}(y)},
\end{equation*}
where $p(d)>0$ gives the spatial scaling of the approximations $\Phi_{N}$ and $\tilde{\Phi}_{N}$ (conditions on $p(d)$ raising to convergence are given in Remark \ref{rmk:propagation:of:chaos} above and proved below). Since $\Phi$ is compactly supported, there exists $R>0$ such that $\supp(\Phi)\subset B(0,R)$ which in particular implies that $\supp(\Phi_{N})\subset B(x,RN^{-p(d)})$ and $\supp(\tilde{\Phi}_{N})\subset B(\tilde{x},RN^{-p(d)})$. Until the end of the proof, we will assume that $N$ is large enough so that $RN^{-p(d)}\leq r$. Furthermore, we have
\begin{equation}
\max(\|\Phi_{N}\|_{\infty}, \|\tilde{\Phi}_{N}\|_{\infty}) \leq N^{dp(d)} \frac{\|\Phi \|_{\infty}}{\varepsilon}=\alpha_{N},
\end{equation}
and, by applying the quotient rule,
\begin{equation*}
\max\left( \| \nabla \Phi_{N}(y) \|, \| \nabla \tilde{\Phi}_{N}(y) \| \right) \leq N^{dp(d)} \left[ \frac{N^{p(d)} \| \nabla \Phi\|_{\infty}}{\varepsilon} + \frac{\|\Phi\|_{\infty} M}{\varepsilon^2}\right]=\beta_{N}.
\end{equation*}
We are now in position to define suitable sequences $(g_{N})_{N\geq 1}$ and $(\tilde{g}_{N})_{N\geq 1}$ by
\begin{equation*}
g_{N}(\eta,y)= \varphi(\eta) \Phi_{N}(y) \ \text{ and }\ \tilde{g}_{N}(\eta,y)= \tilde{\varphi}(\eta) \tilde{\Phi}_{N}(y).
\end{equation*}
Obviously, the functions $\beta_{N}^{-1} g_{N}$ and $\beta_{N}^{-1} \tilde{g}_{N}$ belong to $Lip_{1}$ and $\max(\| g_{N}\|_{\infty}, \|\tilde{g}_{N}\|_{\infty}) \leq \alpha_{N}$. 

On the one hand, applying the inequality obtained in the first step to $\beta_{N}^{-1} g_{N}$ and $\beta_{N}^{-1} \tilde{g}_{N}$, we deduce that
\begin{equation*}
\left|E\left[\left\langle g_{N}, P^{(N,N)}_{[0,T]}\right\rangle \left\langle \tilde{g}_{N}, P^{(N,N)}_{[0,T]}\right\rangle \right]-\left\langle g_{N}, P_{[0,T]}\right\rangle \left\langle \tilde{g}_{N}, P_{[0,T]}\right\rangle\right|
\end{equation*}
is upperbounded by $2\beta_{N} \alpha_{N} d_{KR}\left(P^{(N,N)}_{[0,T]},P_{[0,T]}\right)$.

On the other hand, we have
\begin{equation*}
\left\langle g_{N}, P_{[0,T]}\right\rangle = \int_{\mathbb{R}^{d}} E\left[\varphi(\bar{Z}_{y})\right] N^{dp(d)}\Phi(N^{p(d)}(y-x)) dy \to \int_{\mathbb{R}^{d}} E\left[\varphi(\bar{Z}_{y})\right] \delta_{x}(dy),
\end{equation*}
thanks to the continuity of $y\mapsto E[\varphi(\bar{Z}_{y})],$ which is a consequence of \eqref{eq:continuity:spatial:marginal:of:P} proven below. Therefore,
\begin{equation*}
\left\langle g_{N}, P_{[0,T]}\right\rangle \to \left< \varphi, P_{[0,T]}(\cdot|x) \right> \ \text{ and }\ \left\langle \tilde{g}_{N}, P_{[0,T]}\right\rangle \to \left< \tilde{\varphi}, P_{[0,T]}(\cdot|\tilde{x}) \right>.
\end{equation*}
Gathering the steps above, we deduce \eqref{Propagation_of_chaos} provided that $\beta_{N} \alpha_{N} d_{KR}\left(P^{(N,N)}_{[0,T]},P_{[0,T]}\right)$ goes to $0$, which holds true if $p(d)<[(4+d)(2d+1)]^{-1}$ under Scenario (\hyperref[ass:random:position]{$S_{1}$}) (apply Theorem \ref{thm:1}) or if $p(d)<[(2\vee d)(2d+1)]^{-1}$ under Scenario (\hyperref[ass:deterministic:position]{$S_{2}$}) (apply Theorem \ref{thm:2}).
\end{proof}

We close this section with the following observation. 
If for each $t\geq 0$ and $x\in \R^d$ we call
\begin{equation}
\label{def:membrane_potential_at_limit}
u(t,x)= e^{-\alpha t}u_{0}(x) +\int_{\R^d}w(y,x)\int_{0}^te^{-\alpha(t-s)}\lambda(s,y)ds\rho(dy),
\end{equation}
then clearly $\lambda(t,x)=f(u(t,x))$ and $u(t,x)$
satisfies the {\it scalar neural field equation}
\begin{equation}
\label{def:neural_field_equation}
\begin{cases}
\dfrac{\partial u(t,x)}{\partial t}=-\alpha u(t,x)+ \int_{\R^d}w(y,x)f(u(t,y))\rho(dy)\\
u(0,x)= u_{0}(x).
\end{cases}
\end{equation}
Writing $ U^{(N)} ( t, x_i ) := U^{(N)}_i (t) ,$ where $U^{(N)} _i ( t) $ has been defined in \eqref{eq:intensity} above, we obtain the convergence of $ U^{(N)} (t, x_i ) $ to the solution $u(t,x) $ of the neural field equation in the following sense. 

\begin{corollary}\label{cor:convergencetoneuralfield}
Under the conditions of either Theorem \ref{thm:1} or Theorem \ref{thm:2}, we have that 
\begin{equation}
\lim_{N \to \infty } E \left( \int_\R  \int_0^T  | U^{(N)} (t,x) - u(t,x) | dt \mu^{(N)} (dx) \right) = 0 ,
\end{equation}
for any $T > 0, $ where expectation is taken with respect to the randomness present in the jumps of the process.
\end{corollary} 
The proof of this corollary goes along the lines of the proof of Theorem \ref{thm:1} and Theorem \ref{thm:2}, in Sections \ref{sec:5} and \ref{sec:6} below.

\section{Estimating $d_{KR}(P^{(N,N)}_{[0,T]},P_{[0,T]})$ for fixed positions $x_1, \ldots, x_N $}\label{sec:4}
Assume that the following quantities are given and fixed:
\begin{itemize}
\item the number of neurons $N$,
\item the positions of the neurons $x_{1},\dots,x_{N}$.
\end{itemize}

Hereafter the following empirical measures will be used (the first and the last one are just reminders of \eqref{def:marginal_empirical_measure} and \eqref{eq:def:limit:measure} respectively),
\begin{equation}\label{eq:definition:three:measures}
\begin{cases}
P^{(N, N)}_{[0,T]}(d\eta,dx)=\frac{1}{N}\sum_{i=1}^{N}\delta_{\left((Z^{(N)}_i(t))_{0\leq t\leq T}, x_i\right)}(d\eta,dx),\\
P^{(\infty , N)}_{[0,T]}(d\eta,dx)=P_{[0,T]}(d\eta|x)\mu^{(N)}(dx), \\
P_{[0,T]}(d\eta,dx)=P_{[0,T]}(d\eta|x)\rho (dx).
\end{cases}
\end{equation}

To estimate $d_{KR}(P^{(N,N)}_{[0,T]},P_{[0,T]})$ we shall proceed as follows. We will first show that $P^{(N,N)}_{[0,T]}$ and $P^{(\infty,N)}_{[0,T]}$ are close to each other by using a suitable coupling. The rate of convergence of such a coupling is a balance between the variance coming from the $N$ particles and the bias induced by the replacement of $\rho(dx)$ by $\mu^{(N)}(dx)$. Next, it will be shown that the $d_{KR}$-distance between $P^{(\infty,N)}_{[0,T]}$ and $P_{[0,T]}$ is controlled in terms of the Wasserstein distance between $\mu^{(N)}(dx)$ and $\rho(dx)$. %Then, an estimate for $d_{KR}(P^{(N,N)}_{[0,T]},P_{[0,T]})$ will be %inferred from the triangle inequality.

\subsection{Estimating the $d_{KR}$-distance between $P^{(N, N)}_{[0,T]}$ and $P^{(\infty, N)}_{[0,T]}$}
The aim of this subsection is to upper-bound $d_{KR}(P^{(N, N)}_{[0,T]},P^{(\infty, N)}_{[0,T]})$ when the positions $x_1, \ldots, x_N\in \R^d $ are fixed. 

\begin{thm}
\label{thm:direct:coupling}
Under Assumptions \ref{ass:1}, \ref{Ass:2} and \ref{ass:w:lipschitz}, for each $N$ in $\N$ and $T>0$ there exists a constant $C=C(\alpha,f,w,T,u_{0})>0$ such that for a fixed choice $ x_1, \ldots , x_N\in \R^d $ of positions, \begin{multline}\label{eq:control:P:N:N:P:infty:N}
d_{KR}\left(P^{(N, N)}_{[0,T]}, P^{(\infty, N)}_{[0,T]} \right) + E \left( \int_\R \int_0^T | U^{(N)} (t, x) - u(t, x) | dt \mu^{(N)} (dx) \right) \\
\leq CN^{-1/2} \left[
\left( \int_{\R^d} \left(\int_{0}^T\lambda(t,x)dt\right)^2\mu^{(N)}(dx)\right)^{1/2}\right. \\+ \left.
\left(\sup_{j } \| w_{x_j} \|_{L^2(\mu^{(N)})}\exp\left(\sup_{j} \| w_{x_j} \|_{L^1(\mu^{(N)})}T\right)+1\right)\left(\int_{\R^d}\int_0^T \lambda(t,x)dt\mu^{(N)}(dx)\right)^{1/2} \right] \\ 
+C W_2 ( \mu^{(N)} , \rho )   \exp\left(\sup_{j } \| w_{x_j} \|_{L^1(\mu^{(N)})} T\right).
\end{multline}
\end{thm}

\begin{proof}
Fix a test function $g\in Lip_1$ and observe that by definition
\begin{equation*}
\left| \left\langle g, P^{(N,N)}_{[0,T]}-P^{(\infty, N)}_{[0,T]} \right\rangle \right| = \left| \frac{1}{N}\sum_{i=1}^{N} \left[g\left((Z^{(N)}_i(t))_{0\leq t\leq T},x_i\right) - \int g\left(\eta,x_i\right) P_{[0,T]}(d\eta|x_{i})\right] \right|.
\end{equation*}
We will introduce in Equation \eqref{eq:def:limit:process} below a suitable coupling between the processes $(Z^{(N)}_i(t))_{0\leq t\leq T}$, $i=1,\dots,N$ and the processes $(\bar{Z}_i(t))_{0\leq t\leq T}$, $i\geq 1$, the latter being independent and distributed according to $P_{[0,T]}(d\eta|x_{i})$. In particular, we can decompose the equation above as
\begin{multline*}
\left| \left\langle g, P^{(N,N)}_{[0,T]}-P^{(\infty, N)}_{[0,T]} \right\rangle \right| \leq \left| \frac{1}{N} \sum_{i=1}^{N}\left[ g\left((Z^{(N)}_i(t))_{0\leq t\leq T},x_i\right) - g\left((\bar{Z}_i(t))_{0\leq t\leq T},x_i\right)\right] \right| \\
+ \left| \frac{1}{N} \sum_{i=1}^{N}\left[ g\left((\bar{Z}_i(t))_{0\leq t\leq T},x_i\right) -  \int g\left(\eta,x_i\right) P_{[0,T]}(d\eta|x_{i}) \right]\right|\leq A^{N}(T)+ B^{N}(T),
\end{multline*}
with 
\begin{equation*}
\begin{cases}
A^{N}(T):= \frac{1}{N} \sum_{i=1}^{N} \sup_{0\leq t\leq T} \left| Z^{(N)}_{i}(t)- \bar{Z}_{i}(t) \right| \\
B^{N}(T):=\left| \frac{1}{N} \sum_{i=1}^{N} G_{i} - E[G_{i}]\right|,
\end{cases}
\end{equation*}
where $G_{i}:= g\left((\bar{Z}_i(t))_{0\leq t\leq T},x_i\right)$ for each $i\in\{1,\ldots, N\}$. To obtain the upper bound $A^N(T)$ we have used the $1$-Lipschitz continuity of $g$ and the inequality $d_S(\eta,\xi)\leq \sup_{t\leq T}|\eta(t)-\xi(t)|$ which is valid for all $\eta,\xi\in D([0,T],\N).$

Thus it suffices to obtain upper bounds for the expected values of $A^{N}(T)$ and $B^{N}(T)$. We start studying $B^{N}(T)$. By the Cauchy-Schwarz inequality and the independence of the $\bar{Z}_{i}$'s, it follows that
$$
E[B^N(T)]\leq \left[\frac{1}{N^2}\sum_{i=1}^N \mbox{Var}[G_i] 
\right]^{1/2},$$
so that we only need to control the variance of each $G_{i}$. Now, let $(\tilde{Z}_i(t))_{0\leq t\leq T}$ be an independent copy of $(\bar{Z}_i(t))_{0\leq t\leq T}$ and set  $\tilde{G}_i=g\left((\tilde{Z}_i(t))_{0\leq t\leq T},x_i\right)$. In what follows, the expectation $\tilde{E}$ is taken with respect to\ $\tilde{G}_i$. Thus, by applying Jensen's inequality we deduce that
$$
\mbox{Var}(G_i)=E\left[\left(G_i-E[G_i]\right)^2\right]=E\left[
\left[\tilde{E} (G_i-\tilde{G}_i)\right]^2\right]\leq E\left[\tilde{E}\left[(G_i-\tilde{G}_i)^2\right]\right] .
$$  
Then, since $d_S(\eta,\xi)\leq \sup_{0\leq t\leq T}|\eta(t)-\xi(t)|$ for all $\eta,\xi \in D([0,T],\N)$ and both processes $(\bar{Z}_i(t))_{0\leq t\leq T}$ and $(\tilde{Z}_i(t))_{0\leq t\leq T}$ are increasing, the $1$-Lipschitz continuity of $g$ implies that
$$
|G_i-\tilde{G}_i|\leq \sup_{0\leq t\leq T}|\bar{Z}_i(t)-\tilde{Z}_i(t)|\leq \bar{Z}_i(T)+\tilde{Z}_i(T) . 
$$
Therefore, by applying once more Jensen's inequality we obtain that 
$$
\tilde{E}\left[(G_i-\tilde{G}_i)^2\right]\leq 2\tilde{E}\left[\left(
\bar{Z}_i(T)^2+\tilde{Z}_i(T)^2\right)\right]
=2\left(\bar{Z}_i(T)^2+E\left[\bar{Z}_i(T)^2\right]
\right).
$$
Collecting all the estimates we then conclude that
$$
\mbox{Var}(G_i)\leq 4 E\big[\bar{Z}_i(T)^2\big].
$$
Now, noticing that $\bar{Z}_i(T)$ is a Poisson random variable with rate $\int_0^T\lambda(t,x_i)dt ,$ 
\begin{equation*}
E\left[\bar{Z}_{i}(T)^{2}\right] = \mbox{Var}(Z_i(T))+\left(E[\bar{Z}_i(T)]\right)^2 =  \int_{0}^{T} \lambda(t,x_{i})dt + \left(\int_{0}^{T} \lambda(t,x_{i}) dt \right)^{2}.
\end{equation*}
Hence, we have just shown that
$$
E[B^N(T)]\leq 2N^{-1/2}
\left[\int_{\R^d}\int_{0}^T\lambda(t,x)dt\mu^{(N)}(dx)+\int_{\R^d}
\left(\int_{0}^T\lambda(t,x)dt\right)^2\mu^{(N)}(dx)\right]^{1/2}.
$$ 
Since clearly $(u+v)^{1/2}\leq u^{1/2}+v^{1/2}$ for all $u,v\geq 0$, it follows that
\begin{multline}
E[B^N(T)]\leq 2N^{-1/2} \left[
\left(\int_{\R^d}\int_{0}^T\lambda(t,x)dt\mu^{(N)}(dx)\right)^{1/2} \right. \\ \left. + \left(\int_{\R^d}
\left(\int_{0}^T\lambda(t,x)dt\right)^2\mu^{(N)}(dx)\right)^{1/2} \right].
\end{multline}
 
We shall next deal with $A^{N}(T)$.
Let us now introduce the coupling we consider here. Let $(\Pi_i(dz,ds))_{i\geq 1}$ be a sequence of i.i.d.\ Poisson random measures with intensity measure $dsdz$ on $\R_+\times \R_+$. By Proposition \ref{prop:equal_definition} the process $(Z^{(N)}_i(t))_{t\geq 0,i=1,\dots,N}$ defined for each $t\geq 0$ and $i\in\{1,\ldots, N\}$ by
\begin{equation}
Z^{(N)}_i(t)=\int_{0}^t\int_{0}^{\infty} \one_{\Big\{z\leq f\Big( e^{-\alpha s}u_{0}(x_{i})+ \frac{1}{N} \sum_{j=1}^{N} w\left(x_{j},x_{i}\right)\int_{0}^{s} e^{-\alpha(s-h)}dZ^{(N)}_j(h) \Big)\Big\}}\Pi_i(dz,ds) ,
\end{equation}
is also a multivariate nonlinear Hawkes process with parameters $(N,f,w,u_{0},\alpha),$ and the processes $(\bar{Z}_i(t))_{t\geq 0,i=1,\dots,N}$ defined for each $i\in\{1,\ldots, N\}$ and $t\geq 0$ as
\begin{equation} \label{eq:def:limit:process}
\bar{Z}_i(t)= \int_{0}^t\int_{0}^{\infty} \one_{ \Big\{z\leq f\Big( e^{-\alpha s}u_{0}(x_{i}) +\int_{\R^d}w(y,x_{i})\int_{0}^{s} e^{-\alpha(s-h)}\lambda(h,y)dh\rho(dy) \Big)\Big\}}\Pi_i(dz, ds),
\end{equation}
are independent and such that $(\bar{Z}_i(t))_{t\geq 0}$ is distributed according to $P(\cdot|x_{i})$ for each $i\in\{1,\ldots, N\}$.
Now, for each $i$ and $t\geq 0$, let us define the following quantity
$$
\Delta^{(N)}_i(t)=\int_{0}^t |d(Z^{(N)}_i(s)-\bar{Z}_i(s))|.
$$
Using successively that $f$ is Lipschitz and the triangle inequality we deduce that
\begin{multline}\label{eq:tobecite}
E[\Delta^{(N)}_i(T)]\leq L_f E \left( \int_0^T | U^{(N)}(s, x_i)  - u(s, x_i) | ds \right) \\
:= L_f \left( F^{(N)}_{i}(T)+G^{(N)}_{i}(T)+H^{(N)}_{i}(T) \right) ,
\end{multline}
where $ U^{(N)} ( s, x_i ) := U^{(N)}_i (s) ,$ with  $U^{(N)} _i ( s) $ as in \eqref{eq:intensity}, and
\begin{equation*}
\begin{cases}
F^{(N)}_{i}(T):= E\left[\left| \int_0^{T} \frac{1}{N} \sum_{j=1}^{N} w\left(x_j,x_i\right) \int_{[0, t [ }e^{-\alpha(t-s)} \left[ dZ^{(N)}_j(s) - d\bar{Z}_{j}(s) \right]dt \right|\right],\\
G^{(N)}_{i}(T):= E\left[ \left| \int_0^{T} \frac{1}{N} \sum_{j=1}^{N} w\left(x_j,x_i\right) \int_{[0, t [ }e^{-\alpha(t-s)} \left[ d\bar{Z}_j(s) - \lambda(s,x_{j})ds \right]dt \right| \right],\\
H^{(N)}_{i}(T):= \left|  \int_0^{T}\int_{[0, t[ }e^{-\alpha(t-s)} \Big[ \frac{1}{N} \sum_{j=1}^{N} w(x_{j},x_{i}) \lambda(s,x_{j})- \int_{\R^d} w(y,x_{i})\lambda(s,y) \rho(dy) \Big]ds  dt \right| .
\end{cases}
\end{equation*}
Notice that since $e^{-\alpha(s-h)}\leq 1$ for $0\leq h\leq s$, we have
\begin{equation*}
F^{(N)}_{i}(T)\leq \frac{1}{N} \sum_{j=1}^{N} |w(x_{j},x_{i})| \int_{0}^{T} E[\Delta^{(N)}_j(s)] ds ,
\end{equation*}
which in turn implies that
\begin{equation}
\frac{1}{N}\sum_{i=1}^N F^{(N)}_{i}(T)\leq \sup_j \|w_{x_j} \|_{L^1(\mu^{(N)})} \int_{0}^{T}\frac{1}{N}\sum_{j=1}^{N} E[\Delta^{(N)}_j(s)] ds .
\end{equation}
Now, write $W_{j}(t):=  w\left(x_j,x_i\right) \int_{[0, t [ }e^{-\alpha(t-s)} d\bar{Z}_j(s) $. By the triangle inequality, we have
$$G^{(N)}_{i}(T)\leq  E\left[ \int_{0}^{T} \left| \frac{1}{N} \sum_{j=1}^{N} W_{j}(t) - E[W_{j}(t)] \right| dt \right].$$
Then the Cauchy-Schwarz inequality and the independence of the $W_{j}$'s implies that 
$$
G^{(N)}_{i}(T)\leq \int_{0}^T \left(\frac{1}{N^2}\sum_{j=1}^N \mbox{Var}(W_j(t))\right)^{1/2}dt.
$$
Now, $\textrm{Var}(W_{j}(t)) = w(x_{j},x_{i})^{2} \int_{0}^{t} e^{-2\alpha(t-s)}\lambda(s,x_{j})ds,$ whence
$$
\frac{1}{N}\sum_{i=1}^N G^{(N)}_i(T)\leq \int_0^T \frac{1}{N} 
\sum_{i=1}^N\left(\frac{1}{N^2}\sum_{j=1}^{N}w(x_{j},x_{i})^{2} \int_{0}^{t} e^{-2\alpha(t-s)}\lambda(s,x_{j})ds\right)^{1/2}dt.
$$
Applying Jensen's inequality twice, the right-hand side of the inequality above can be upper-bounded by % in the following line I changed the w^xi 
$$
\sup_j \|w_{x_j}\|_{L^2(\mu^{(N)})}T\left(\frac{1}{T}\frac{1}{N^2}\sum_{j=1}^{N}\int_0^T  \int_{0}^{t} e^{-2\alpha(t-s)}\lambda(s,x_{j})dsdt\right)^{1/2}.
$$
Since $\int_0^T  \int_{0}^{t} e^{-2\alpha(t-s)}\lambda(s,x_{j})dsdt=\int_0^T\lambda(t,x_j)(2\alpha)^{-1}(1-e^{-2\alpha(T-t)})dt\leq T\int_0^T \lambda(s,x_j)dt$ (recall that $1-e^{-v}\leq v$ for all $v\geq 0$), we deduce that % in the following line I changed the w^xi 
$$
\frac{1}{N}\sum_{i=1}^N G^{(N)}_i(T)\leq \sup_j \|w_{x_j}\|_{L^2(\mu^{(N)})}T\left(\int_{\R^d}\int_0^T \lambda(t,x)dt\mu^{(N)}(dx)\right)^{1/2}N^{-1/2}.
$$
Finally we shall deal with $H^{(N)}_{i}(T)$. Proceeding similarly as above, we have
\begin{multline*}
H^{(N)}_{i}(T)\leq \int_{0}^{T} \int_{0}^{t} e^{-\alpha(t-s)} \left|  \int_{\R^d} w(y,x_i)\lambda(s,y)\left[\mu^{(N)}(dy)-\rho(dy)\right] \right| ds dt\\
\leq T\int_{0}^{T} \left|  \int_{\R^d}w(y,x_i)\lambda(t,y)\left[\mu^{(N)}(dy)-\rho(dy)\right] \right| dt ,
\end{multline*}
and therefore it follows that
$$
\frac{1}{N}\sum_{i=1}^N H^{(N)}_{i}(T)\leq T\int_0^{T}\int_{\R^d} \left|\int_{\R^d}w(y,x)\lambda(t,y)\left[\mu^{(N)}(dy)-\rho(dy)\right]\right|\mu^{(N)}(dx)dt.
$$
Thus, defining $\delta(t)=N^{-1}\sum_{i=1}^N E[\Delta_i(t)]$ and then gathering the steps above gives % in the following line I changed the w^xi 
\begin{multline*}
\delta(T)\leq T \left( \sup_j \|w_{x_j} \|_{L^2(\mu^{(N)})} \left(\int_{\R^d}\int_0^T \lambda(t,x)dt\mu^{(N)}(dx)\right)^{1/2}N^{-1/2}\right.\\  
+ \left. \int_0^{T}\int_{\R^d} \left|\int_{\R^d}w(y,x)\lambda(t,y)\left[\mu^{(N)}(dy)-\rho(dy)\right]\right|\mu^{(N)}(dx)dt\right.\\ \left.
+ \sup_j \|w_{x_j} \|_{L^1(\mu^{(N)})} \int_{0}^{T} \delta(t)dt\right).\\
\end{multline*}
Now, Gronwall's lemma implies
\begin{multline}
\delta(T)\leq T \left( \sup_j \|w_{x_j} \|_{L^2(\mu^{(N)})} \left(\int_{\R^d}\int_0^T \lambda(t,x)dt\mu^{(N)}(dx)\right)^{1/2}N^{-1/2} \right. \\ \left. +\int_0^{T}\int_{\R^d} \left|\int_{\R^d}w(y,x)\lambda(t,y)\left[\mu^{(N)}(dy)-\rho(dy)\right]\right|\mu^{(N)}(dx)dt \right) \\
 \exp\left(\sup_j \|w_{x_j} \|_{L^1(\mu^{(N)})} T\right).
\end{multline}
To deduce \eqref{eq:control:P:N:N:P:infty:N}, it suffices to observe that $E[A^{N}(T)]\leq \delta(T)$ and the following control proven below: there exists a constant $C=C(\alpha,f,w,T,u_{0})>0$ such that 
\begin{equation}\label{eq:control}
\int_0^{T}\int_{\R^d} \left|\int_{\R^d}w(y,x)\lambda(t,y)\left[\mu^{(N)}(dy)-\rho(dy)\right]\right|\mu^{(N)}(dx)dt \le C W_2 ( \mu^{(N)} , \rho ) .
\end{equation}

Indeed, let $ ( \tilde \Omega, \tilde{ \cal A}, \tilde P ) $ be a probability space  on which are defined random variables $ Y_1^{(N)} $ and $ Y_2^{(N) }   $ such that their joint law under $\tilde P $ is the optimal coupling achieving the Wasserstein distance $W_2 ( \mu^{(N)}, \rho ) $ of order $2$ between $ \mu^{(N)} $ and $ \rho .$
Then for all $ t \le T, $ using H\"older's inequality,  
\begin{multline}\label{eq:lal}
 \left|\int_{\R^d}w(y,x)\lambda(t,y)\left[\mu^{(N)}(dy)-\rho(dy)\right]\right| \\
 = \left| \tilde E ( w( Y_1^{(N)} , x)\lambda ( t, Y_1^{(N) } ) - w(Y_2^{(N) } , x) \lambda (t, Y_2^{(N) }  ) \right| \\
\le \| \lambda \|_{ [0, T ] \times \R^d  , \infty }  \tilde E ( | w( Y_1^{(N)} , x) - w( Y_2^{(N) } , x)| ) \\
+ \| w^x  \|_{L^2 (\rho ) } \left( \tilde E ( | \lambda (t, Y_1^{(N)}  ) - \lambda (t, Y_2^{(N) }  ) |^2 ) \right)^{1/2 } .
\end{multline} 
By Assumption \ref{ass:w:lipschitz}, 
$$ \tilde E ( | w( Y_1^{(N)} , x) - w( Y_2^{(N) } , x)| ) \le  L_w  \tilde E ( |  Y_1^{(N)}  - Y_2^{(N) } | )  \le L_w  W_2 ( \mu^{(N)} , \rho ) .$$ 
Moreover, using \eqref{lambda_holder_L1}, 
$$  \left( \tilde E ( | \lambda (t, Y_1^{(N)}  ) - \lambda (t, Y_2^{(N) }  ) |^2 ) \right)^{1/2 }   \le  C W_2 ( \mu^{(N)} , \rho ) ,$$ 
implying \eqref{eq:control}.
Finally, observe that together with \eqref{eq:tobecite}, we obtain the same control for $ E \left( \int_\R \int_0^T | U^{(N)} (t, x) - u(t, x) | dt \mu^{(N)} (dx) \right). $
\end{proof}

\subsection{Estimating the $d_{KR}$-distance between $P^{(\infty,N)}_{[0,T]}$ and $P_{[0,T]}$}
In this subsection we give an upper bound for $d_{KR}(P^{(\infty, N)}_{[0,T]},P_{[0,T]})$ in terms of the Wasserstein distance between the empirical distribution $\mu^{(N)}(dx)$ and the limit measure $\rho(dx)$. 

\begin{proposition}
\label{prop:last:control}
Grant Assumptions \ref{ass:1}--\ref{ass:w:lipschitz}. For each $N\geq 1$ and $T>0$, there exists a positive constant $C=C(f,u_0,w,\alpha,T)$ such that for any choice of $ x_1, \ldots , x_N\in \R^d$ the following inequality holds
\begin{equation}\label{eq:control:P:infty:N:P}
d_{KR}\left(P^{(\infty, N)}_{[0,T]}, P_{[0,T]}\right)\leq C W_1(\mu^{(N)},\rho) ,
%\int \int \left( \int_{0}^{T} |\lambda(t,x)-\lambda(t,y)| dt + |x-y|\right) W^N ( dx, dy ) ,   
\end{equation}
where $W_1(\mu^{(N)},\rho)$ is the Wasserstein distance between $\mu^{(N)}(dx)$ and $\rho(dx)$ associated with the metric $d(x,y)=\|x-y\|$ for $ x, y \in \R^d .$  
\end{proposition}

\begin{proof}
Notice that for deterministic probability measures $P$ and $\tilde P $ which are defined on $D([0,T],\N)\times \R^d$, the distance $d_{KR}\left(P, \tilde P \right)$ reduces to 
\begin{equation}
\label{def:classical_KR_distance}
d_{KR}\left(P, \tilde P \right)=\sup_{g\in Lip_1 , \| g \|_\infty < \infty }\left\langle g, P- \tilde P \right\rangle.  
\end{equation}
Now, fix a test function $g\in Lip_1.$ We take any coupling $ W^{(N)}(dx, dy ) $ of $\mu^{(N)} ( dx)$ and $ \rho (dy) .$ Given $ x$ and $y$ in $\mathbb{R}^{d}$, we use the canonical coupling of $\bar{Z}_x (t) $ and $ \bar{Z}_y ( t)$. That is, we pose 
$$ \bar Z_x ( t) =  \int_{0}^t\int_{0}^{\infty} \one_{ \Big\{z\leq f\Big( e^{-\alpha s}u_{0}(x) +\int_{\R^d}w(y,x)\int_{0}^{s} e^{-\alpha(s-h)}\lambda(h,y)dh\rho(dy) \Big)\Big\}}\Pi(dz, ds),
$$
and 
$$ \bar Z_y ( t) =  \int_{0}^t\int_{0}^{\infty} \one_{ \Big\{z\leq f\Big( e^{-\alpha s}u_{0}(y) +\int_{\R^d}w(y',y)\int_{0}^{s} e^{-\alpha(s-h)}\lambda(h,y')dh\rho(dy') \Big)\Big\}}\Pi(dz, ds),
$$
where we use the same Poisson random measure $ \Pi (dz, ds)  $ for the construction of $ \bar Z_x (t) $ and of $ \bar Z_y ( t) .$
Then
\begin{multline*}
\left\langle g, P_{[0,T]}^{(\infty, N)}-P_{[0,T]} \right\rangle = E \left( \int_{\R^d }   g( \bar Z_x (t) ,x) \mu^{(N)} (dx) -  \int_{\R^d } g (\bar Z_y (t) ,y) \rho (dy) \right) \\
= E  \int_{\R^d }  \int_{\R^d }\left[g( \bar Z_x (t),x ) -  g (\bar Z_y (t),y)\right] W^{(N)} ( dx, dy ) ,
\end{multline*}
since $ \mu^{(N)} ( \R^d ) = \rho ( \R^d ) = 1 .$ By the Lipschitz continuity of $g,$
$$ \left|\left\langle g, P_{[0,T]}^{(\infty, N)}-P_{[0,T]} \right\rangle\right| \le E \! \int \!\!\!\! \int \left[ d_S ( (\bar Z_x (t))_{0 \le t \le T} , (\bar Z_y (t) )_{0 \le t \le T} ) + \|x-y\| \right] W^N (dx, dy).$$
Yet, as a consequence of the canonical coupling, it follows that
\begin{equation}\label{eq:continuity:spatial:marginal:of:P}
E\left[d_S ( \bar Z_x , \bar Z_y )\right] \leq E\left[\sup_{[0,T]} |\bar Z_x (t) -\bar Z_y (t)|\right] \leq \int_{0}^{T} |\lambda(t,x)-\lambda(t,y)| dt\leq CT||x-y||,
\end{equation}
where we used \eqref{lambda_holder_L1}. Finally, the assertion follows from the definition of $W_{1}$.
\end{proof}

Before going to the proofs of Theorems \ref{thm:1} and \ref{thm:2}, let us sum up the results obtained above and then present the scheme of the remaining proofs. 

\begin{remark}
Combining Theorem \ref{thm:direct:coupling} and Proposition \ref{prop:last:control} together with the inequalities  $W_1(\mu^{(N)},\rho)\leq W_2(\mu^{(N)},\rho)$ and $\| \cdot \|_{L^1(\mu^{(N)})}\leq \|\cdot \|_{L^2(\mu^{(N)})}$, it follows that to conclude the proofs of Theorems  it suffices to control as $N\to+\infty$:\\
 1- the Wasserstein distance $W_2(\mu^{(N)},\rho)$; 2- the supremum $\sup_{j } \| w_{x_j} \|_{L^2(\mu^{(N)})}$.
\end{remark}
To treat the second point of the remark above, we use the following technical lemma in order to use \eqref{eq:assumption:w:2}.

\begin{lemma}\label{lem:control:L2:by:W1:compact:support}
Grant Assumption \ref{ass:w:lipschitz}. Let $\mu$, $\rho$ be two probability measures. Assume that $\rho$ satisfies Assumption \ref{ass:exponential:moments} and that $\mu$ is supported in $B(0^{d},r)$ for some $r>0$. Then, for any $\beta'<\beta$, there exists a constant $C=C(w,\beta,\beta')$ such that
\begin{equation}\label{eq:control:L2:W1:by:truncation}
\sup_{||y||\leq r}  \left| ||w_{y}||^2_{L^{2}(\mu)}-||w_{y}||^2_{L^{2}(\rho)} \right| \leq C \big(1+ r \big) \, W_{1}(\mu,\rho) + C\mathcal{E}_{\beta} e^{-\beta' r}.
\end{equation}
\end{lemma}
\begin{proof}
The proof is based on a truncation argument. Let us define the auxiliary measure $\rho_{r}(dx)$ by
\begin{equation}\label{eq:truncated:rho}
\rho_{r}(dx):= \rho(dx)\mathbf{1}_{B(0^{d},r)}(x) + \left(1-\rho(B(0^{d},r))\right)\delta_{0^{d}}(dx),
\end{equation}
which is the truncated version of $\rho(dx)$ to which we add a Dirac mass at $0^{d}$. 

The Lipschitz continuity of $w$ implies that for any $y,x,z\in B(0^d,r),$
$$
|w_y^2(x)-w_y^2 (z)|\leq C_w(1+r)\|x-z\|, 
$$
that is, $x\in B(0^{d},r)\mapsto (w_y (x))^{2}$ is Lipschitz with explicit constant, where the positive constant $C_w$ depends only on $w(0,0)$ and $L_w$. We hence deduce that for all $y\in B(0^d,r),$
\begin{equation*}
\left| ||w_{y}||^2_{L^{2}(\mu)}-||w_{y}||^2_{L^{2}(\rho_{r})} \right| =\left| \int_{B(0^{d},r)} w_{y}(x)^{2} d(\mu-\rho_{r})(x) \right| \leq C_{w} (1+r) W_{1}(\mu,\rho_{r}),
\end{equation*}
thanks to the Kantorovich-Rubinstein duality. By the triangle inequality it follows then 
\begin{equation*}
\left| ||w_{y}||^2_{L^{2}(\mu)}-||w_{y}||^2_{L^{2}(\rho_{r})} \right| \leq C_{w} (1+r) (W_{1}(\mu,\rho)+W_{1}(\rho,\rho_{r})).
\end{equation*}
By the canonical coupling, we have
$$
W_{1}(\rho,\rho_{r})\leq \int_{\|x\|>r}\|x\|\rho(dx).
$$
Now, since for any $\beta'<\beta$, there exists a constant $C>0$ depending only on $\beta$ and $\beta'$ such that $\|x\|e^{\beta'\|x\|}\leq C e^{\beta\|x\|}$ we infer that $\int_{\|x\|>r}\|x\|\rho(dx)\leq C\mathcal{E}_{\beta}e^{-\beta'r}$ which implies that  $W_{1}(\rho,\rho_{r})\leq C\mathcal{E}_{\beta}e^{-\beta'r}$.

Now, with $M_{r}=|w(0,0)| + L_wr$, we have for any $y\in B(0^d,r)$,
\begin{eqnarray*}
\left| ||w_{y}||^2_{L^{2}(\rho_{r})}-||w_{y}||^2_{L^{2}(\rho)} \right| & \leq & \int_{||x||>r} \left(M_{r} + L_w||x||\right)^{2} \rho(dx) + \left(1-\rho(B(0^{d},r))\right) M_{r}^{2}\\
& \leq &  2L^2_w\int_{||x||>r} ||x||^{2} \rho(dx) +  3\left(1-\rho(B(0^{d},r))\right) M_{r}^{2}.
\end{eqnarray*}
On the one hand, for any $\beta'<\beta$, $\int_{||x||>r} ||x||^{2} \rho(dx)\leq \int_{||x||>r} ||x||^{2} e^{\beta'||x||} \rho(dx) e^{-\beta' r}$ and there exists a constant $C$ that only depends on $\beta$ and $\beta'$ such that $||x||^{2} e^{\beta'||x||}\leq C e^{\beta||x||}$. Hence $\int_{||x||>r} ||x||^{2} \rho(dx)\leq C \mathcal{E}_{\beta} e^{-\beta' r}$.
On the other hand, $1-\rho(B(0^{d},r))= \int_{||x||>r} \rho(dx) \leq \mathcal{E}_{\beta}e^{-\beta r}$. The same argument applies and gives the existence of a constant $C$ such that $\left( 1-\rho(B(0^{d},r)) \right)M_{r}^{2}\leq C \mathcal{E}_{\beta} e^{-\beta' r}$. Finally, \eqref{eq:control:L2:W1:by:truncation} follows from triangular inequality.
\end{proof}

\section{Proof of Theorem \ref{thm:1}}\label{sec:5}
In this section we give the proof of Theorem \ref{thm:1} using the estimates obtained in the previous section.
We work under the additional assumption that the positions $x_{1},\dots,x_{N}$ are realizations of i.i.d.\ random variables $X_1, \ldots , X_N, $  distributed according to $\rho .$ 

On the one hand, to control the Wasserstein distance, Theorem 1.6 \ Item (i) of \cite{Bolley} gives that for any fixed $ d' > d$ there exist constants $K$ and $N_0$ depending only on $d$, $\beta$ and $\mathcal{E}_{\beta}$ such that  
$$P ( W_2 ( \mu^{(N)} , \rho )  > \varepsilon ) \le e^{ - K N \varepsilon^2 } ,$$
for any $0<\varepsilon\leq 1$ and $N\geq N_0\max\{1,\varepsilon^{-(4+d')}\}$. As a consequence, for any fixed $ d' > d$ it follows that if $ \varepsilon_N = O ( N^{ - \frac{1}{4 + d'}}) $ then 
$$\sum_{N=1}^{\infty}P ( W_2 ( \mu^{(N)} , \rho )  > \varepsilon_N )<\infty,$$
so that Borel-Cantelli's lemma implies 
\begin{equation}\label{eq:rate:convergence:W2:random}
W_{2}(\mu^{(N)},\rho)\leq CN^{-\frac{1}{4+d'}}
\end{equation}
eventually almost surely.

On the other hand, define $R_N=N^{\gamma}$ where the constant $\gamma>0$ will be specified later and observe that by Markov's inequality,
$P\left(\cup_{i=1}^N\left\{\|X_i\|> R_N\right\}\right)\leq N P(\|X_1\|> R_N)\leq \mathcal{E}_{\beta}e^{-\beta R_N(1-o(1))}$. As a consequence, 
$\sum_{N=1}^{\infty}P\left(\cup_{i=1}^N\left\{\|X_i\|> R_N\right\}\right)<\infty$ and by Borel-Cantelli's lemma we deduce that for almost all realizations $X_1,X_2,\ldots$ there exists a $N_0$ depending on that realization such that for all $N\geq N_0$, $\mu^{(N)}$ is supported in $B(0^{d},R_{N})$.

Taking $d'=d+1$, $0<\gamma<1/(10+2d)$ and applying Lemma \ref{lem:control:L2:by:W1:compact:support} gives that almost surely,
\begin{equation*}
\sup_{||y||\leq R_{N}}  \left| ||w_{y}||^2_{L^{2}(\mu^{(N)})}-||w_{y}||^2_{L^{2}(\rho)} \right| \xrightarrow[N\to \infty]{} 0,
\end{equation*}
and in particular,
\begin{equation}\label{eq:uniform:control:discrete:L2}
\limsup_{N\to\infty }\sup_{1\leq j\leq N}\|w_{x_j}\|^2_{L^2(\mu^{(N)})}\leq \sup_{y\in\R^d}\int_{\R^d}(w_y(x))^2\rho(dx)<\infty.
\end{equation}

Finally, the first two assertions of Theorem \ref{thm:1} follow immediately from \eqref{eq:control:P:N:N:P:infty:N} and \eqref{eq:control:P:infty:N:P} combined with \eqref{eq:uniform:control:discrete:L2} -- or $\sup_{j}\|w_{x_j}\|_{L^2(\mu^{(N)})}\leq ||w||_{\infty}$ in case $w$ is bounded; \eqref{eq:secondbound} is a consequence of \eqref{eq:rate:convergence:W2:random}.

\section{Proof of Theorem \ref{thm:2}}\label{sec:6}
In this section we give the proof of Theorem \ref{thm:2} using the estimates obtained in Section \ref{sec:4}.
The main issue here is to construct a sequence of empirical distributions $\mu^{(N)}$ for which the Wasserstein distance $W_2 ( \mu^{(N)} , \rho ) $ is controlled.

We first use a truncation argument to reduce to the case where $\rho$ is compactly supported.
Let $r>0$ be a truncation level to be chosen later. Let us define the measure $\rho_{r}$ by
\begin{equation*}
\rho_{r}(dx):= \rho(dx)\mathbf{1}_{B(0^{d},r)}(x) + \left(1-\rho(B(0^{d},r))\right)\delta_{0^{d}}(dx).
\end{equation*}
By the canonical coupling, we have
\begin{equation*}
W_{2}(\rho,\rho_{r})\leq \int_{||x||>r} ||x||^{2} \rho(dx).
\end{equation*}
Yet, if $r$ is large enough, 
\begin{equation*}
\int_{||x||>r} ||x||^{2} \rho(dx) \leq \int_{||x||>r} e^{\beta ||x||} \rho(dx) r^{2} e^{-\beta r},
\end{equation*}
hence, by Assumption \ref{ass:exponential:moments}, there exists $\beta>0$ such that
\begin{equation}\label{eq:control:W2:true:truncated}
W_{2}(\rho,\rho_{r})\leq \mathcal{E}_{\beta} r^{2}e^{-\beta r}<+\infty.
\end{equation}

Let us now describe how we construct an empirical distribution made of $N$ points adapted to any probability measure $\rho_{r}$ supported in $B(0^{d},r)$. For simplicity of our construction, we assume that $\mathbb{R}^{d}$ is endowed with the distance induced by the  $\ell^{\infty}$ norm.

The construction is iterative, so we first explain how each step works. Let $\nu$ be a measure having support in the cube $B(0^{d},r)=[-r,r]^{d}$ and denote by $|\nu|\geq 1/N$ its mass. Let us prove that
\begin{equation}
\text{\parbox{.85\textwidth}{there exists a cube $C$ with $\nu(C)\geq 1/N$ and radius less than $r \lfloor \left(N|\nu|\right)^{1/d} \rfloor^{-1}\leq 2r (N|\nu|)^{-1/d}$.}}
\end{equation}

\begin{sloppypar}
Assume that the $\nu$-mass of any cube of radius equal to $r \lfloor \left(N|\nu|\right)^{1/d} \rfloor^{-1}$ is less than $1/N$. There exists a covering of the cube $[-r,r]^{d}$ into $ \lfloor \left(N|\nu|\right)^{1/d} \rfloor^{d}$ disjoint smaller cubes, each one of radius equal to $r \lfloor \left(N|\nu|\right)^{1/d} \rfloor^{-1}$. This implies 
\begin{equation*}
|\nu| < \lfloor \left(N|\nu|\right)^{1/d} \rfloor^{d} N^{-1} \leq ((N|\nu|)^{1/d})^{d} N^{-1}= |\nu|
\end{equation*}
yielding a contradiction. Then, treat separately the cases $\left(N|\nu|\right)^{1/d}\geq 2$ and $\left(N|\nu|\right)^{1/d}<2$ to prove $r \lfloor \left(N|\nu|\right)^{1/d} \rfloor^{-1}\leq 2r (N|\nu|)^{-1/d}$.
\end{sloppypar}

Applying the iterative step above to the probability measure $\rho_{r}$ gives the existence of a cube $C_{N}$ such that $\rho_{r}(C_{N})\geq 1/N$ and $\diam(C_{N})\leq 4 r N^{-1/d}$. Then, we define the measure
\begin{equation*}
\rho_{r}^{N}:= \frac{N^{-1}}{\rho_{r}(C_{N})} \rho_{r} \mathbf{1}_{C_{N}}.
\end{equation*}
Its mass is $1/N$. Applying the iterative step to $\tilde{\rho}_{r}=\rho_{r}-\rho_{r}^{N}$ (its mass is $(N-1)/N$) gives a cube $C_{N-1}$ such that $\tilde{\rho}_{r}(C_{N-1})\geq 1/N$ and $\diam(C_{N-1})\leq 4 r (N-1)^{-1/d}$. Similarly we define $\rho_{r}^{N-1}:= \frac{N^{-1}}{\tilde{\rho}_{r}(C_{N-1})} \tilde{\rho}_{r} \mathbf{1}_{C_{N-1}}$. In brief, applying $N$ times the iterative step gives a sequence of cubes $C_{1},\dots,C_{N}$ and associated measures $\rho_{r}^{1},\dots,\rho_{r}^{N}$ such that for all $k$, 
\begin{equation*}
\diam(C_{k})\leq 4 r  k^{-1/d},
\end{equation*}
$\rho_{r}^{k}$ is a measure of mass $1/N$ supported in $C_{k}$, and $\rho_{r}=\sum_{k=1}^{N} \rho_{r}^{k}$. 

For each $k$, let $x_{k}$ denote the center of $C_{k}$ and let $\mu^{(N)}=N^{-1}\sum_{k=1}^{N}\delta_{x_{k}}$ denote the associated empirical distribution. To control $W_{2}(\mu^{(N)},\rho_{r})$, we use the canonical coupling $\pi(dx,dy)=\sum_{k=1}^{N} \left(N^{-1}\delta_{x_{k}}\right)\otimes \rho_{r}^{k}$. Hence,
\begin{multline*}
W_{2}(\mu^{(N)},\rho_{r})^{2} \leq N^{-1} \sum_{k=1}^{N} \int_{\mathbb{R}^{d}} \int_{\mathbb{R}^{d}} ||x-y||^{2} \delta_{x_{k}}(dx)\rho_{r}^{k}(dy)\\
\leq N^{-1} \sum_{k=1}^{N} \diam(C_{k})^{2} \leq 16r^{2}N^{-1}\sum_{k=1}^{N} k^{-2/d}.
\end{multline*}
If $d=1$, then $\sum_{k=1}^{+\infty} k^{-2/d}=\pi^{2}/6$ so $W_{2}(\mu^{(N)},\rho_{r})^{2}\leq g_{1}(r,N):= (4\pi^{2}r/6)N^{-1/2}$. If $d\geq 2$, by H\"older's inequality\footnote{This is not optimal (see \cite{chevallier2018uniform} for a refined version of this quantization argument).},
\begin{equation*}
\sum_{k=1}^{N}k^{-2/d}\leq \left(\sum_{k=1}^{N} k^{-1}\right)^{2/d} N^{1-2/d}\leq N \left(\frac{1+ \ln N}{N}\right)^{2/d},
\end{equation*}
so that
\begin{equation}\label{eq:control:W2:empirical:truncated}
W_{2}(\mu^{(N)},\rho_{r})\leq g_{d}(r,N):= 4 r \left(\frac{1+ \ln N}{N}\right)^{1/d}.
\end{equation}

We now chose a truncation level that depends on $N$, namely $r_{N}=N^{\varepsilon}$ for some $\varepsilon>0$. Combining \eqref{eq:control:W2:true:truncated} and the results above with the triangular inequality, we have, for $N$ large enough,
\begin{equation}\label{eq:control:W2:deterministic:case:truncation}
W_{2}(\mu^{(N)},\rho)\leq CN^{2\varepsilon}e^{-\beta N^{\varepsilon}} + g_{d}(N^{\varepsilon},N).
\end{equation}
Hence, for any $d'>2\vee d$, there exist $K,N_{0}$ such that for all $N\geq N_{0}$,
\begin{equation}\label{eq:control:W2:deterministic:case:final}
W_{2}(\mu^{(N)},\rho)\leq K N^{-1/d'}.
\end{equation}

Taking $d'=d+2$, $0<\varepsilon<1/(d+2)$ and applying Lemma \ref{lem:control:L2:by:W1:compact:support} gives that
\begin{equation*}
\sup_{||y||\leq r_{N}}  \left| ||w_{y}||^2_{L^{2}(\mu^{(N)})}-||w_{y}||^2_{L^{2}(\rho)} \right| \xrightarrow[N\to +\infty]{} 0, 
\end{equation*}
and in particular,
\begin{equation}\label{eq:uniform:control:discrete:L2:deterministic}
\limsup_{N\to\infty }\sup_{1\leq j\leq N}\|w_{x_j}\|^2_{L^2(\mu^{(N)})}\leq \sup_{y\in\R^d}\int_{\R^d}w^2_y(x)\rho(dx)<+\infty.
\end{equation}

Finally, the first assertion of Theorem \ref{thm:2} follows immediately from \eqref{eq:control:P:N:N:P:infty:N} and \eqref{eq:control:P:infty:N:P} combined with \eqref{eq:uniform:control:discrete:L2:deterministic}; \eqref{eq:secondbound:deterministic} is a consequence of \eqref{eq:control:W2:deterministic:case:final}.

\begin{proof}[Proof of Corollary \ref{cor:convergencetoneuralfield}]
Corollary \ref{cor:convergencetoneuralfield} follows under both scenarios from \eqref{eq:control:P:N:N:P:infty:N} in Theorem \ref{thm:direct:coupling}, together with the arguments used to conclude the proofs of Theorem \ref{thm:1} and \ref{thm:2}.
\end{proof}

\section{Final discussion}
\label{Sec:Discussion_on_the_parameters}

In the previous sections, some technical assumptions have been imposed regarding the parameters of our model. These parameters are: the spike rate function $f$, the initial condition $u_{0}$ and the matrix of synaptic strengths $w.$ Here, we discuss these assumptions with respect to standard choices appearing e.g.\ in \cite{Bressloff:12}.

There are three main choices for the function $f$: a sigmoid-like function, a piecewise linear function or a Heaviside function  (see page 6 of \cite{Bressloff:12}). The first two choices obviously fit our Lipschitz condition (Assumption \ref{ass:1}). A Heaviside function does of course not satisfy the Lipschitz condition. However, a Heaviside nonlinearity is less realistic and is mainly studied for purely mathematical reasons (to obtain explicit computations).

A typical choice for the function $u_{0}$ is a Gaussian kernel  (see page 38 of \cite{Bressloff:12}). It can describe an initial bump of the neural activity at some location of the cortex.

Usually, a (homogeneity) simplification is made concerning the function $w$: $w(y,x)$ is assumed to depend on $\|x-y\|$ only. Under this simplification, a common choice is the so-called Mexican hat function (see page 41 of \cite{Bressloff:12}). Nevertheless, inhomogeneous neural fields where the previous simplification is dropped are also studied (see section 3.5 of \cite{Bressloff:12} for instance). As a consequence of the modeling used in the present article, the interaction strength felt by neurons at position $x$ coming from neurons in the vicinity $dy$ of $y$ is given by $w(y,x)\rho(dy)$. Hence, inhomogeneity in neural networks can be considered in two ways:
\begin{itemize}
\item with an inhomogeneous matrix $w$ (whereas the limit spatial distribution $\rho$ is homogeneous),
\item with an homogeneous $w$ but an inhomogeneous distribution $\rho$.
\end{itemize}

Let us discuss the spatial distribution $\rho$. Two standard choices are a uniform distribution over a bounded set (see \cite{luccon2014mean}), or a finite sum of Dirac masses (in that case, the present paper is highly related to \cite{SusanneEva}). In these cases, $\rho$ is compactly supported and therefore satisfies Assumption \ref{ass:exponential:moments}. For unbounded domains, a typical choice is a Gaussian distribution satisfying Assumption \ref{ass:exponential:moments} as well. Finally, let us notice that uniform distributions over unbounded domains (such as the Lebesgue measure on $\mathbb{R}^{d}$) are also considered for the study of the neural field equation. Such a model (which cannot correspond to a probability distribution on the positions) does not fit our assumptions and therefore cannot be obtained as the limit of a microscopic description following our approach.\\

Finally, the rate of convergence obtained in the present paper depends on the modeling scenario: the positions of the neurons are random or they are deterministic. Notice that our approach gives better rates of convergence in the deterministic framework (rate in $N^{-\frac{1}{2\vee d'}}$ for any $d'>d$) than in the random one (rate in $N^{-\frac{1}{4+d'}}$ for any $d'>d$).

\section*{Acknowledgements}
This research has been conducted as part of the project Labex MME-DII (ANR11-LBX-0023-01); it is part of USP project {\em Mathematics, computation, language
and the brain}, FAPESP project {\em Research, Innovation and
Dissemination Center for Neuromathematics} (grant 2013/07699-0), CNPq projects {\em Stochastic modeling of the brain activity}
(grant 480108/2012-9) and {\em Plasticity in the brain after a brachial plexus lesion} (grant 478537/2012-3).
AD and GO are fully supported by a FAPESP fellowship (grants  2016/17791-9  and 2016/17789-4  respectively). 

\bigskip
\appendix

\section{Remaining Mathematical Proofs}\label{sec:App}

\subsection{Proof of Proposition \ref{Prop:2}}
\label{proof of Prop.2}
Using the inequality $f(u)\leq L_f|u|+f(0)$ valid for all $u\in \R$ (which follows from the Lipschitz continuity of $f$),
we have that for each $1\leq i\leq N$,
\begin{equation}\label{eq:Afirst}
\lambda^{(N)}_i(t) \leq f(0)+L_f\left(e^{-\alpha t}|u_0(x_i)|+\frac{1}{N}\sum_{j=1}^N|w(x_j,x_i)| \int_{[0,t[}e^{-\alpha(t-s)}dZ^{(N)}_j(s) \right).
\end{equation}

Thus, using that $e^{-\alpha(t-s)}\leq 1$ for all $0\leq s\leq t$,  we obtain that
\begin{eqnarray*}
E\left[\lambda^{(N)}_i(t)\right]\leq f(0)+L_f\left(|u_0(x_i)|+\frac{1}{N}\sum_{j=1}^N|w(x_j,x_i)|E\left[Z^{(N)}_j(t)\right]\right).
\end{eqnarray*}
Then, denoting $\beta(t)= N^{-1}\sum_{i=1}^N E\left[Z^{(N)}_i(t)\right]$ for each $t\geq 0$,  it follows that
\begin{multline}\label{ineq:expct_value_Z_j}
\beta(T) =  \frac{1}{N}\sum_{i=1}^N\int_{0}^TE\left[\lambda^{(N)}_i(t)\right] dt\\
 \leq  T\left(f(0)+L_f\int_{\R^d}|u_0(x)|\mu^{(N)}(dx)\right) + L_f\sup_j \|w_{x_j}\|_{L^1(\mu^{(N)})}
\int_{0}^T \beta(t)dt.
\end{multline} 
Proposition \ref{prop:existence} implies that $t\mapsto \beta(t)$ is locally bounded so that the first inequality stated in Proposition \ref{Prop:2} follows from Gronwall's inequality. We now turn to the control of the second moment of $Z^{(N)}_i (t) .$ We first work with the stopped processes $Z_i^{(N)} ( \cdot \wedge \tau_K ) , $ where $ \tau_K = \inf \{ t \geq 0 : \sum_{i=1}^N Z_i^{(N)} (t) \geq K\}  ,$ for some fixed truncation level $ K > 0 .$ By It\^o's formula, 
\begin{multline*}
 E \left[ (Z^{(N)}_i(t\wedge \tau_K))^2  \right] = E \left[ Z^{(N)}_i(t\wedge \tau_K )\right]   + 2  E \left[ \int_0^{t \wedge \tau_K} Z^{(N)}_i(s) \lambda^{(N)}_i(s)ds\right] \\
 \le  E \left[ Z^{(N)}_i(t)\right] + \int_0^t E \left[ (Z^{(N)}_i(s\wedge \tau_K))^2 \right] ds +  \int_0^t E \left[  (  \lambda^{(N)}_i(s))^2\right] ds .
 \end{multline*}
\eqref{eq:Afirst} implies that 
\begin{eqnarray*}
[\lambda^{(N)}_i(t)]^2 \leq 2 f(0)^2 +4 L_f^2 \left(|u_0(x_i)|^2 +\frac{1}{N}\sum_{j=1}^N|w(x_j,x_i)|^2 [Z^{(N)}_j(t)]^2 \right).
\end{eqnarray*}
Denoting $\gamma_K (t)= N^{-1}\sum_{i=1}^N E\left[(Z^{(N)}_i(t\wedge \tau_K))^2\right]$ for each $t\geq 0,$ it follows that
\begin{multline}\label{ineq:expct_value_square}
\gamma_K (T) \le \beta ( T) + \int_0^T\gamma_K (s) ds + 
2  T f(0)^2 +\\
 4 L_f^2  T \int_{\R^d}|u_0(x)|^2 \mu^{(N)}(dx)  + 4 L_f^2 \left(\sup_j \|w_{x_j}\|_{L^2(\mu^{(N)})}\right)^2
\int_{0}^T \gamma_K(s)ds.
\end{multline} 
This implies, applying once more Gronwall's inequality, that
\begin{multline*}
\frac{1}{N}\sum_{i=1}^N E\left[(Z^{(N)}_i(T\wedge \tau_K ))^2 \right]\leq  \exp \left\{  T \left( 1 + 4 L_f^2 \left(\sup_j \|w_{x_j}\|_{L^2(\mu^{(N)})}\right)^2 \right) \right\} \times  \\
\times \left[ T\left(f(0)+L_f \| u_0\|_\infty \right)\exp\left\{TL_f \sup_j \|w_{x_j} \|_{L^1(\mu^{(N)})} \right\} + 2 T f(0)^2 + 4 L_f^2 T\| u_0\|_\infty^2 \right] 
.
\end{multline*}
We now obtain the result by letting $K \to \infty .$ 

\subsection{Proof of Proposition \ref{prop:apriori}}\label{proof of Propapriori}
Taking the test function $ g (\eta , x) := \eta_T $ which belongs to $Lip_1, $ we obtain  first for any $N, $ 
\begin{multline*}
\int_{\R^d } \int_{ D ( \R_+ , \N ) }  g(\eta) P_{[0, T ]} ( d \eta , dx)  \le  d_{KR} ( P_{[0, T ]}^{(N,N)} , P_{[0, T ]} )  \\
+  E  \left( \int_{\R^d } \int_{ D ( \R_+ , \N ) }  g(\eta) P^{(N, N)} _{[0, T ]} ( d \eta , dx) \right)  ,
\end{multline*}   
and then, letting $N \to \infty $ and using \eqref{eq:28}, 
$$
\int_{\R^d } \int_{ D ( \R_+ , \N ) }  g(\eta) P_{[0, T ]} ( d \eta , dx)   \le  \limsup_{ N \to \infty } E  \left( \int_{\R^d } \int_{ D ( \R_+ , \N ) }  g(\eta) P^{(N, N)} _{[0, T ]} ( d \eta , dx) \right) .
$$
Yet
\begin{equation*}
\limsup_{ N \to \infty } E  \left( \int_{\R^d } \int_{ D ( \R_+ , \N ) }  g(\eta) P^{(N, N)} _{[0, T ]} ( d \eta , dx) \right)  =  \limsup_{ N \to \infty } \frac{1}{N} \sum_{i=1}^N E [ (Z_i^{(N)} ( T) ] < \infty 
\end{equation*}   
by Proposition \ref{Prop:2} together with \eqref{eq:uniform:control:discrete:L2} or \eqref{eq:uniform:control:discrete:L2:deterministic} (depending on the chosen scenario).

To prove the second assertion, fix a truncation level $K > 0$ and let $ \Phi_K : \R_+ \to \R_+ $ be a smooth bounded function such that 
$ \Phi_K ( x) = x^2 $ for all $ x \le K, $ $\Phi_K ( x) \le x^2 $ for all $ x \geq 0$  and such that $ \| \Phi_K\|_{Lip} := \sup_{ x \neq y } \frac{ |\Phi_K ( x)  - \Phi_K (y ) |}{ |x-y|} \le 4 K .$ Put then $ g ( \eta ) := \frac{1}{4K } \Phi_K ( \eta_T) ,$ by construction, this function belongs to $Lip_1.$ As before, we obtain that  
$$
\int_{\R^d } \int_{ D ( \R_+ , \N ) }  g(\eta) P_{[0, T ]} ( d \eta , dx)  \le 
 \limsup_{ N \to \infty } E  \left( \int_{\R^d } \int_{ D ( \R_+ , \N ) }  g(\eta) P^{(N, N)} _{[0, T ]} ( d \eta , dx) \right)  ,
$$
which implies, multiplying $ g (\eta) $ by $4K, $ that 
$$
\int_{\R^d } \int_{ D ( \R_+ , \N ) }  \Phi_K (\eta_T) P_{[0, T ]} ( d \eta , dx)  \le 
 \limsup_{ N \to \infty } E  \left( \int_{\R^d } \int_{ D ( \R_+ , \N ) }   \eta_T^2 P^{(N, N)} _{[0, T ]} ( d \eta , dx) \right)  ,
$$
where we have used that $ \Phi_K ( x) \le x^2$ to obtain the rhs which does not depend on $K$ any more.
As in the first step of the proof,  the rhs of the above inequality is finite, thanks to Proposition \ref{Prop:2} together with \eqref{eq:uniform:control:discrete:L2} or \eqref{eq:uniform:control:discrete:L2:deterministic} (depending on the chosen scenario). Therefore, letting now $K \to \infty$ in the lhs of the above equation, the assertion follows. 

\subsection{Proof of Proposition \ref{prop:uniqueness}}
\label{proof of Prop.3}
Using the inequality $f(u)\leq L_f|u|+f(0)$ valid for all $u\in \R$, one gets for each $0\leq t\leq T $,
\begin{equation*}
\lambda(t,x)\leq f(0)+L_f\left[e^{-\alpha t}|u_0(x)|+ \int_{\R^d}|w(y,x)|\int_{0}^t e^{-\alpha(t-s)}\lambda(s,y)ds\rho(dy) \right].
\end{equation*}
We use that $e^{-\alpha(t-s)}\leq 1$ for all $0\leq s\leq t, $ that $u_0$ is bounded and apply  H\"older's inequality to the last term (with respect to $\rho (dy ) $) to obtain the upper bound 
\begin{equation*}
\lambda(t,x)\leq f(0)+L_f\|u_0\|_{\infty}+ L_f \| w^x \|_{L^2 ( \rho ) }  \int_{\R^d}\left( \int_{0}^t\lambda(s,y)ds\right)^2  \rho(dy).
\end{equation*}
Since $t\to \int_{\R^d}(\int_{0}^t\lambda(s,y)ds)^2\rho(dy)$ is locally bounded by assumption and since by Assumption \ref{ass:w:lipschitz} together with Remark \ref{rem:2}, $x \mapsto \| w^x \|_{L^2 ( \rho ) }$ is bounded, one obtains from the inequality above that $ \lambda (t, x) $ is  bounded on $ [0, T ] \times \R^d ,$ for any fixed $ T > 0 .$ 

We now prove the continuity of the function $\lambda(t,x)$. 
To that end, take $t,t'\in [0,T]$ with $t\leq t' \leq T$ and $x,y$ in $\mathbb{R}^{d}$. On the one hand, by using successively the Lipschitz continuity of $f$, the triangle inequality and the boundedness of $u_0,$ we deduce that
\begin{multline*}
|\lambda(t,x)-\lambda(t',x)|\leq L_f\|u_0\|_{\infty}|e^{-\alpha t}-e^{-\alpha t'}|\\ + L_f\int_{\R^d}| w(y, x )| \int_{0}^{t'}|e^{-\alpha(t-s)}-e^{-\alpha(t'-s)}|\lambda(s,y)ds\rho(dy)\\
+L_f \int_{\R^d} | w (y, x) | \int_{t}^{t'} e^{-\alpha(t-s)}\lambda(s,y)ds\rho(dy).
\end{multline*}  
Write $ \| \lambda \|_{ [0, T ] \times \R^d , \infty } := \sup_{x \in \R^d , t \le T } \lambda ( t, x ) $ which is bounded thanks to the first step of the proof. It follows from the inequality above that 
\begin{multline}\label{eq:time:continuity:lambda}
|\lambda(t,x)-\lambda(t',x)|\leq \alpha L_f\|u\|_{\infty}h+L_f\|w^x \|_{L^1 ( \rho ) } \| \lambda \|_{ [0, T ] \times \R^d , \infty } \\
\left( (e^{\alpha h}-1)(1-e^{-\alpha h})+\frac{(e^{\alpha h}-1)}{\alpha}\right).
\end{multline}
On the other hand, the Lipschitz continuity of $f$ implies that
\begin{multline*}
|\lambda(t',x)-\lambda(t',y)|\leq L_f e^{-\alpha t' }| u_0(x) - u_0(y)|\\
+L_f\int_{0}^{t'} e^{-\alpha(t'-s)} \int_{\R^d}\lambda(s,z)|w(z,x)-w(z,y)|ds \rho(dz).
\end{multline*}
Thus, using the Lipschitz-continuity of $u_0 ,$ the Lipschitz-continuity of $w$ and the boundedness of $\lambda$, we deduce from the inequality above that 
\begin{equation}\label{eq:lambda:holder}
|\lambda(t',x)-\lambda(t',y)|\leq L_f\left( e^{-\alpha t' }L_{u_0}+ \| \lambda \|_{ [0, T] \times \R^d, \infty } (1-e^{-\alpha t'})\alpha^{-1}L_w\right)\|x-y\|.
\end{equation}
Inequality \eqref{eq:lambda:holder} proves \eqref{lambda_holder_L1} and, together with \eqref{eq:time:continuity:lambda}, proves the continuity of $\lambda$.

Therefore it remains to establish \eqref{ineq:uniqueness}. For that sake, observe that the Lipschitz continuity of $f$ implies that
$$
|F(\lambda)(t,x)-F(\tilde{\lambda})(t,x)|\leq L_f\int_{\R^d}|w(y,x)|\int_{0}^t e^{-\alpha(t-s)} \left|\lambda(s,y)-\tilde{\lambda}(s,y)\right|ds \rho (dy),
$$
whence
$$
|F(\lambda)(t,x)-F(\tilde{\lambda})(t,x)|\leq L_f\|w^x \|_{L^1 ( \rho ) } \|\lambda-\tilde{\lambda}\|_{[0,T]\times \R^d,\infty}\int_{0}^t e^{-\alpha(t-s)}ds ,
$$
which implies the result.     
\\\\

\newpage
\bibliographystyle{abbrv}
\bibliography{Bibli}

\end{document}